\documentclass[11pt,a4paper]{article}
\usepackage[utf8]{inputenc}
\usepackage[T1]{fontenc}
\usepackage{lmodern}
\usepackage{a4wide} 
\usepackage{graphicx} 

\usepackage{amsmath}
\usepackage{amsfonts}
\usepackage{amssymb}
\usepackage{amsthm}
\usepackage[numbers]{natbib}
\usepackage{color}

\usepackage{enumerate} 
\usepackage{hyperref}
\usepackage{todonotes} 

\usepackage{fourier} 

\theoremstyle{plain}
\newtheorem{thm}{Theorem}[section]  
\newtheorem{prop}[thm]{Proposition}
\newtheorem{cor}[thm]{Corollary}
\newtheorem{lem}[thm]{Lemma}

\theoremstyle{definition}

\theoremstyle{remark}
\newtheorem*{rem}{Remark}

\numberwithin{equation}{section}

\DeclareMathOperator{\integers}{\mathbb{Z}}
\DeclareMathOperator{\reals}{\mathbb{R}}

\DeclareMathOperator{\nat}{\mathbb{N}}

\DeclareMathOperator{\Prob}{\mathbb{P}}

\DeclareMathOperator{\ind}{\mathbf{1}}

\newcommand{\norm}[1]{\left\|#1\right\|}

\newcommand{\set}[1]{\left\{ #1 \right\}}

\newcommand{\Latd}{\mathbb{Z}^d}

\title{Law of large numbers for random walks\\ on attractive spin-flip dynamics}

\author{Stein Andreas Bethuelsen \footnotemark[1]
\\  Markus Heydenreich \footnotemark[2] }

\footnotetext[1]{Technische Universit\"at M\"unchen, Fakult\"at f\"ur Mathematik, Boltzmannstr. 3, 85748 Garching, Germany. Email: stein.bethuelsen@tum.de}
\footnotetext[2]{Mathematisches Institut, Universit\"at M\"unchen, Theresienstr. 39, 80333 M\"unchen, Germany. Email: m.heydenreich@lmu.de}

\begin{document}

\maketitle

\begin{abstract}
We prove a law of large numbers for certain random walks on certain attractive dynamic random environments when initialised from all sites equal to the same state. This result applies to random walks on $\Latd$ with $d\geq1$. We further provide sufficient mixing conditions under which the assumption on the initial state can be relaxed, and obtain estimates on the large deviation behaviour of the random walk. 

As prime example we study the random walk on the contact process, for which we obtain a law of large numbers in arbitrary dimension. For this model, further properties about the speed are derived.
\end{abstract}

\vspace{0.5cm}

\bigskip



\section{Introduction and main results} 
\subsection{Background and outline}
Random walks in random environment (RWRE) gained much interest throughout the last de\-cades. Such models serve as natural extensions of the classical random walk model and have broad applications in physics, chemistry and biology.

RWRE models show significantly different behaviours than the simple random walk model. This was already observed in one of the first models studied, \citet{SolomonRWRE1975}, where it was shown that the random walk can behave sub-ballistically. Non-Gaussian scaling limits were established for the same model in \citet*{KestenKozlovSpitzerRWRE1975} and \citet{SinaiRWRE1982}. These characteristics are due to trapping phenomena.

RWRE models on $\integers$ are by now well understood in great generality whenever the environment is static, i.e.\ it does not change with time. On the other hand, for RWRE models on $\integers^d$, $d\geq 2$, the analysis of trapping phenomena becomes much more delicate and less is known.
See for instance \citet{ZeitouniRWRE2004} or \citet{SznitmanRWRE2004} for an overview of results, and \citet{DrewitzRamirezRWRE2013} for a monograph with focus on recent developments.\newpage

In the last decade, much focus has been devoted to models where the random environment evolves with time, i.e.\ random walks in \emph{dynamic} random environments (RWDRE). It is believed that the extent to which trapping phenomena occur for RWDRE models depends on the correlation structure of the dynamics.

At a rigorous level, it is known to great generality that RWDRE models scale diffusively when the environment is only weakly correlated in space-time; see for instance
\citet{RedigVolleringRWDRE2013}. These results are not restricted to random walks on $\integers$, but are valid in any dimension. Here weakly correlated essentially means that the environment becomes approximately independent of its starting configuration within a space-time cone, also known as cone mixing environment.\

Little is known at a general level when the environment has a non-uniform correlation structure, though trapping phenomena are conjectured to occur for some specific models (\citet{AvenaThomannRWDREsimulations2012}).  \citet*{AvenaHollanderRedigRWDRELDP2009} have shown rigorously that a random walk on the one-dimensional exclusion process exhibits trapping phenomena at the level of large deviations under drift assumptions. On the other hand, several other models with a non-uniform correlation structure have been shown to possess diffusive scaling limits, for example 
\citet*{AvenaSantosVolleringRWDRESSEP2013}, \citet*{HilarioHollanderSidoraviciusSantosTeixeiraRWRW2014},  
den Hollander and dos Santos \cite{HollanderSantosRWCP2013}, \citet{HuveneersSimenhausRWSEP2014} and
 \citet{MountfordVaresDCP2013}.

In this paper we present a strong law of large numbers (SLLN) for random walks on certain attractive (or monotone) interacting particle system (IPS). For this, restrictions on both the random walk and the IPS are required. In particular, we assume that the sites of the IPS take values $0$ or $1$ and that the IPS has a graphical representation coupling which is monotone 
with respect to the initial configuration. One class of IPS satisfying the latter assumption are additive and attractive spin-flip systems.

The SLLN is obtained when the IPS is initialised at time $0$ from a  configuration where all sites have the same value, assuming in addition that the jump transitions of the random walk  only depend on the state of the IPS at the position of the random walk. Under certain mixing conditions, we are able to relax the restriction on the starting configuration. 

An important feature of the SLLN is that it does not rely directly on the correlation structure of the environment, but rather assumes monotonicity. In particular, the SLLN applies to a large class of models with non-uniform correlation structure not previously  considered in the literature.
Furthermore, the SLLN applies to random walks on $\mathbb{Z}^d$ for any $d\geq1$ and is not restricted to nearest neighbour jumps.

We also provide large deviation estimates for the random walk. In particular, we show that no trapping phenomena occur for our model at the level of large deviations throughout the cone mixing regime.

The supercritical contact process is an example of an IPS satisfying the above requirements and having non-uniform correlation structure. For the random walk on this process, we prove the SLLN throughout the supercritical regime when started from the upper invariant measure (as well as many other initial configurations). Further properties about the speed are also derived. These results extend upon the SLLN obtained in \cite{HollanderSantosRWCP2013} beyond the one-dimensional nearest-neighbour setting.

\subsubsection*{Outline}
The rest of the paper is organised as follows. In the next subsection we give a precise definition of our model and in Subsection \ref{sec results Markus} we present our main results. Subsection \ref{sec disc} contains a discussion of related literature.
Section \ref{sec construct} is devoted to a particular coupling construction of the environment and the random walk, yielding a monotonicity property, important for our results. In Section \ref{sec proofs Markus} we present the proofs of our main theorems for general attractive environments. Section \ref{sec contact Markus} is devoted to the special case of a random walk on the contact process.

\subsection{The model}
We first introduce the environment. For this, let $d\geq1$ and denote by $\Omega = \set{0,1}^{\mathbb{Z}^d}$ the configuration space and by $D_{\Omega}[0,\infty)$ the corresponding path space, that is, the set of c\`adl\`ag functions on $[0,\infty)$ taking values in $\Omega$.

As the environment we consider an IPS,  $\xi = (\xi_t)_{t \geq 0}$, such that $\xi_t =\left\{\xi_t(x) \colon x \in \mathbb{Z}^d\right\}$ is in $\Omega$, $t \in [0,\infty)$, and $\xi \in D_{\Omega}[0,\infty)$. 
The process $\xi$ starting from $\xi_0=\eta$ is denoted by $\xi^{\eta}$ and its law is given by $P^{\eta}$. When $\xi_0$ is drawn from $\mu \in \mathcal{P}(\Omega)$, the set of probability measures on $\Omega$, we write $\xi^{\mu}$ for the corresponding process. Its law is denoted by $P^{\mu}$ and is given by
\begin{equation}
P^{\mu}(\cdot)= \int_{\Omega} P^{\eta}(\cdot)\mu(d\eta).
\end{equation}
We assume throughout that $\xi$ is translation invariant, that is,
\begin{equation}\label{eq stationarity}
P^{\eta}(\theta_x \xi_t \in \cdot) = P^{\theta_x \eta}( \xi_t \in \cdot)
\end{equation}
with $\theta_x$ denoting the shift operator $\theta_x\eta(y) =\eta(y-x)$, $\eta \in \Omega$.

Further, to the configuration space $\Omega$ we associate the partial ordering such that $\xi \leq \eta$ with $\xi,\eta \in \Omega$ if and only if $\xi(x) \leq \eta(x)$ for all $x\in \mathbb{Z}^d$. A function $f \colon \Omega \rightarrow \reals$ is called increasing if $\xi\leq\eta$ implies $f(\xi)\leq f(\eta)$.
For two measures $\mu_1,\mu_2 \in \mathcal{P}(\Omega)$, $\mu_2$ \emph{stochastically dominates} $\mu_1$, written $\mu_1 \leq \mu_2$, provided that
\begin{equation}
\int_{\Omega} f d\mu_1 \leq \int_{\Omega} f d\mu_2
\end{equation}
for all increasing continuous functions $f$ on $\Omega$. We denote by $\delta_{\bar{0}},\delta_{\bar{1}}\in \mathcal{P}(\Omega)$  the extremal measures which put all their weight on the configurations $\bar{1}$ and $\bar{0}$, respectively, where $\bar{i}(x)=i$ for all $x \in \mathbb{Z}^d$, $i \in \set{0,1}$. Obviously, it holds that $\delta_{\bar{0}} \leq \delta_{\bar{1}}$.

For a fixed realisation of $(\xi_t)_{t \geq 0}$, let $(W_t)_{t \geq 0}$ be the time-inhomogeneous Markov process on $\mathbb{Z}^d$ that, given $W_t=x$, jumps to $x+z$ at rate $\alpha(\xi_t(x),z)$ for some function $\alpha\colon\{0,1\}\times\mathbb Z^d\to[0,\infty)$. We call this process \emph{the random walk}. Further, we assume throughout that
\begin{equation}\label{eq jump rate walker}
\gamma:= \max_{i \in \{0,1\}} \left\{\alpha(i,o) + \sum_{z \in \mathbb{Z}^d} \norm{z}_1 \alpha(i,z) 
\right\} <\infty,
\end{equation}
where $o \in \mathbb{Z}^d$ denotes the origin.
 Thus, the speed  of the simple random walk seeing only occupied sites ($i=1$) or only vacant sites ($i=0$) is  given by the local drifts 
\begin{equation}\label{eq speed SRW}
u_i := \sum_{z \in \mathbb{Z}^d} \alpha(i,z) z, \quad i\in \set{0,1}.
 \end{equation} 
We say that $(W_t)$ is \emph{elliptic} if, for the unit vectors $\{\pm e_j\}_{j=1,\dots,d}$, we have
  \begin{equation}\label{eq elliptic}
  \min_{i \in \set{0,1}} \min_{j \in \{1,\dots,d\}} \left\{  \alpha(i,\pm e_j) \right\} >0.
  \end{equation} 
  We also say that $(W_t)$ has \emph{finite second moments} if
   \begin{equation}\label{eq finite second mom}
\max_{i \in \set{0,1}} \left\{ \sum_{z \in \mathbb{Z}^d} \alpha(i,z) \norm{z}_1^2 \right\} < \infty
\end{equation} and that it   
    has \emph{finite exponential moments} if there exist $\kappa>0$ such that,
 \begin{equation}\label{eq finite exp mom}
\max_{i \in \set{0,1}} \left\{ \sum_{z \in \mathbb{Z}^d} \alpha(i,z) \exp \left( \epsilon\norm{z}_1 \right) \right\} < \infty, \quad \text{ for all } 0< \epsilon < \kappa.
 \end{equation}
 
 Lastly, for $\xi \in D_{\Omega}[0,\infty)$ and $x\in \mathbb{Z}^d$, let $\Prob^{\xi}_{x}$ denote the law of $(W_t)$ starting from $W_0=x$ in a fixed environment $\xi$, which is the quenched law of $W$. The annealed law of $W$ is given by
 \begin{equation}
 \Prob_{\mu,x}(\cdot) = \int_{ D_{\Omega}[0,\infty) } \Prob_x^{\xi} (\cdot) P^{\mu}(d\xi).
 \end{equation}

\subsection{Main results}\label{sec results Markus}

General IPS can formally be constructed by defining a generator, see \citet[Chapter I.1-3]{LiggettIPS1985}. Alternatively, one can describe an IPS via a countable set of Poisson processes $\mathcal I$, yielding a more probabilistic description, see \citet{DurrettIPS1995}. The probabilistic construction has the advantage that it yields a natural coupling, denoted by $\widehat{P}$, of the dynamics starting from any configuration on a joint probability space. For many interacting particle systems this coupling can be constructed explicitly and is known as the graphical representation.
Important to our approach is the existence of such a coupling $\widehat{P}$  of the dynamic environment $\xi$ which satisfies the following  monotonicity property, 
\begin{align}\label{eq agr}
\widehat{P} \left( \xi_t^{\eta} \leq \xi_t^{\omega}, \: \forall t>0 \text{ and } \eta,\omega \in \Omega \text{ satisfying } \eta \leq \omega \right)=1.
\end{align}
A coupling $\widehat{P}$ satisfying \eqref{eq agr} is said to be an \emph{attractive graphical representation coupling}. 

\begin{thm}[Strong law of large numbers]\label{thm main}
Assume that  $\xi$ has an attractive graphical representation coupling.
 Then, for each $i \in \{0,1\}$, there exists $\rho_i \in [0,1]$ such that
\begin{equation}\label{eq thm main}
\lim_{t \rightarrow \infty} \frac{1}{t}W_t  = \rho_i u_1 + (1-\rho_i) u_0,\quad  \Prob_{\delta_{\bar{i}},o}
\text{-a.s.\ and in } L^1.
\end{equation}
\end{thm}

 Note that Theorem \ref{thm main} does not require $(W_t)$ to be elliptic nor set restrictions on finite range jump transitions, technical assumptions often present in the literature.
Further, \citet[Theorem 2.5]{DurrettIPS1995} yields a large class of IPS with spin-flip dynamics having an attractive graphical representation coupling to which Theorem \ref{thm main} applies. 

The proof of Theorem \ref{thm main} makes use of a particular coupling construction of the random walk together with the sub-additive ergodic theorem. The coupling construction (given in Section \ref{sec construct}) enables us to transfer monotonicity properties of the environment to a functional of $(W_t)$. Informally, this functional counts the number of occupied sites the random walk observes at its jump times, as a function of time. The monotonicity property of this functional together with the graphical representation of the environment naturally leads to a sub-additive structure which we use to obtain a SLLN by employing the sub-additive ergodic theorem (with limit equal to $\rho_i$ as in \eqref{eq thm main}). This is the content of  Theorem \ref{cor main} below. As a last step, the SLLN in Theorem \ref{cor main} is transferred into a SLLN for $(W_t)$, whose proof is given in  Section \ref{sec proof lln}.

The restriction to the extremal starting configurations in Theorem \ref{thm main} can in many cases be relaxed.
For $m > 0$, let 
\begin{equation}\label{eq cone m}
V_m := \set{ (x,t) \in \mathbb{Z}^d \times [0,\infty) : \: \norm{x}_{1} < mt} 
\end{equation}
be a  cone of inclination $m$ opening upwards in space-time. Let $(\xi_t)$ be an IPS with an attractive graphical representation coupling and, for $\eta \in \Omega$, $i \in \{0,1\}$ and $T \in [0,\infty)$, define 
\begin{equation}
\phi_i^{(m)}(\eta,T) :=\widehat{P} \left( \exists (x,t) \in V_m \cap \mathbb{Z}^d \times [T,\infty) : \:\xi_t^{\eta}(x) \neq \xi_t^{\bar{i}}(x) \right).
\end{equation}
Note that $\xi$ is cone mixing, as defined in \cite[Definition 1.1]{AvenaHollanderRedigRWDRELLN2011}, if
\begin{equation}\label{conemixing}
\lim_{ T \rightarrow \infty} \phi_1^{(m)}(\bar{0},T) = 0, \quad \forall \: m \in [0,\infty).
\end{equation}
Further, denote by $\bar{\nu}_0 , \bar{\nu}_1 \in \mathcal{P}(\Omega)$ the ``lower'' and ``upper'' invariant measures  to which $(\eta_t)$ converges when initialised from ${\bar{0}}$ and ${\bar{1}}$ respectively. That is, we have
 $\bar{\nu}_0 = \lim_{t \rightarrow \infty} P^{\bar{0}}(\xi_t \in \cdot) $ and  $\bar{\nu}_1 = \lim_{t \rightarrow \infty} P^{\bar{1}}(\xi_t \in \cdot) $. These limits exist and are invariant under $(\xi_t)$ (see \cite[Theorem III.2.3]{LiggettIPS1985}).
Lastly, recalling \eqref{eq speed SRW}, we denote the convex hull of $u_0$ and $u_1$ by
\begin{equation}
U(u_0,u_1) := \text{conv} \left( u_0,u_1 \right).
\end{equation}

\begin{thm}\label{prop coupling speed}
Assume that  $\xi$ has an attractive graphical representation coupling.
Let $i \in \set{0,1}$ and assume that for some $\epsilon,m>0$ we have $U(u_0,u_1) \times \{1/\gamma\} \subset V_{m(1-\epsilon)}$ and 
\begin{equation}\label{eq conemixing}  \lim_{ T \rightarrow \infty} \phi_i^{(m)}(\eta,T) = 0, \quad\text{for } \bar{\nu}_i-a.e. \: \eta \in \Omega. \end{equation}
Then, for all $\nu\in \mathcal{P}(\Omega)$ such that $ \nu \leq \bar{\nu}_{i}$ (if $i=0$) or $\bar{\nu}_{i}\leq \nu$ (if $i=1$),
\begin{equation}
\lim_{t \rightarrow \infty} \frac{1}{t} W_t = \rho_i u_1 + (1-\rho_i)u_0 \quad \Prob_{\nu,o} \text{-a.s. and in } L^1.
\end{equation}
\end{thm}

Theorem \ref{prop coupling speed} relaxes the assumption on the starting configuration in Theorem \ref{thm main}. Note that \eqref{eq conemixing}  is weaker than cone mixing (in the sense of  \eqref{conemixing}) and applies to IPS with non-uniform correlation structure. The proof of Theorem \ref{prop coupling speed} uses a different coupling construction than that needed for the proof of Theorem \ref{thm main} and is given in Section \ref{sec proofs coupling}.

From the proof of Theorem \ref{thm main}, we also obtain certain large deviation estimates. These are presented in Section \ref{sec proofs ldp}. In particular, we have the following theorem.

\begin{thm}[Large deviation estimates]\label{thm ldp estimate W}
Assume that  $\xi$ has an attractive graphical representation coupling. Further, assume $\rho_0=\rho_1 =: \rho$ (with $\rho_0$ and $\rho_1$ as in  Theorem \ref{thm main}). Consequently, $(W_t)$ satisfies \eqref{eq thm main} $\Prob_{\mu,o}$-a.s with limiting speed $v:= \rho u_1 +(1-\rho)u_0$, irrespectively of $\mu \in \mathcal{P}(\Omega)$. If, in addition, $(W_t)$ has finite exponential moments,
then for any $\epsilon>0$ there exists $C(\epsilon)>0$ such that
\begin{equation}
\Prob_{\mu,o} \left( \norm{W_t-tv}_1 > \epsilon t \right) \leq \exp \left(- C(\epsilon) t\right), \quad \text{ for all } \mu \in \mathcal{P}(\Omega).
\end{equation}
\end{thm}

The proof of Theorem \ref{thm ldp estimate W} is given in  Section \ref{sec proofs ldp}, where we also prove some additional large deviation properties. The key observation for the proof is that the sub-additive structure obtained for the functional used in the proof of Theorem \ref{thm main} yields one-sided large deviations estimates for this functional. The assumption that $\rho_0=\rho_1$ implies two-sided large deviation estimates, however, for the same functional. An additional argument is needed in order to conclude large deviation estimates for $(W_t)$. For this, we use the assumption that $(W_t)$ has finite exponential moments.

\subsubsection*{Random walk on the contact process}\label{sec intro cp}

One classical IPS having an attractive graphical representation coupling is the contact process $\xi = (\xi_t)_{t\geq 0}$. Given $\lambda \in (0,\infty)$, the contact process on $\Latd$ with ``infection rate'' $\lambda$ is defined via its local transition rates, which are given by
\[ \eta \rightarrow \eta^x \text{ with rate } \left\{
	\begin{array}{ll}
		1, & \text{if }\eta(x) =1, \\
		\lambda \sum_{x \sim y} \eta(y), & \text{if } \eta(x)=0.
	\end{array}
\right.\]
Here, $\eta^x$ is defined by $\eta^x(y) := \eta(y)$ for $y \neq x$, and $\eta^x(x) := 1 - \eta(x)$, and $\sum_{x \sim y} $ denotes the sum over nearest neighbours. 

Much is known about the contact process; see \cite[Chapter 1]{LiggettSIS1999} for a thorough introduction. In particular, the empty configuration $\bar{0}$ is an absorbing state for the contact process. On the other hand, starting from the full configuration $\bar{1}$, the contact process evolves towards an equilibrium measure $\bar{\nu}_{\lambda}$, called the ``upper invariant measure'', which is stationary and ergodic with respect to $(\xi_t)$. Further, there is a critical threshold $\lambda_c(d) \in (0,\infty)$, depending on the dimension $d$, such that $\bar{\nu}_{\lambda}=\delta_{\bar{0}}$ for $\lambda \in (0,\lambda_c(d)]$ and, for all $\lambda \in (\lambda_c(d),\infty)$, we have $\bar{\nu}_{\lambda}(\eta(o) =1)>0$.
In particular, for $\lambda > \lambda_c(d)$, the contact process does not satisfy \eqref{conemixing}.

\begin{thm}\label{prop contact process}
Consider the contact process on $\integers^d$ with $d\geq 1$ and infection rate $\lambda \in (\lambda_c(d),\infty)$.
\begin{description}\label{eq lln contact}
\item[a)]  There exists $\rho=\rho(\lambda) \in [0,1]$  such that for all $\nu \in \mathcal{P}(\Omega)$ with $\bar{\nu}_{\lambda} \leq \nu$,
\begin{equation}\label{eq cp lln}
\lim_{t \rightarrow \infty} \frac{1}{t} W_t = u_1 \rho + (1-\rho)u_0  \quad \Prob_{\nu,o} \text{-a.s. and in } L^1.
\end{equation}
\item[b)] The function $\rho\colon (\lambda_c(d), \infty)\to [0,1]$, $\lambda \mapsto \rho(\lambda)$, is non-decreasing and right-continuous in $\lambda$. Moreover, if $(W_t)$ has finite second moments, then
\begin{equation} \label{eq nontrivial rho}
\rho(\lambda) \in (0,1) \quad  
\text{ and }  \quad
\lim_{\lambda \rightarrow \infty} \rho(\lambda) = 1.\end{equation}
\end{description}
\end{thm}

Theorem \ref{prop contact process} extends the law of large numbers of \cite{HollanderSantosRWCP2013}, obtained for the nearest neighbour random walk on the supercritical contact process on $\integers$, to higher dimensions and beyond the nearest neighbour assumption. 

Concerning the proof of Theorem \ref{prop contact process}, note that a) follows immediately from Theorem \ref{thm main} and the graphical representation coupling of the contact process when started from $\bar{1}$. To extend this to any measure stochastically dominated by $\bar{\nu}_{\lambda}$, we prove that the contact process satisfies \eqref{eq conemixing}.

The function $\rho(\cdot)$ in Theorem \ref{prop contact process}b) is the same functional as considered in the proof of Theorem \ref{thm main}. That this is non-decreasing and right-continuous is not difficult to show and follows by monotonicity considerations. Most of the proof of Theorem \ref{prop contact process}b) goes about showing that \eqref{eq nontrivial rho} holds. Note that this result implies that the SLLN in \eqref{eq cp lln} is non-trivial in the sense that the speed of $(W_t)$ is neither $u_0$ nor $u_1$. For the proof of \eqref{eq nontrivial rho}, we treat the two cases $d=1$ and $d\geq2$ separately. For $d=1$ we  extend an argument of  \cite{HollanderSantosRWCP2013} beyond nearest neighbour jumps. 
For $d\geq2$ we use that the supercritical contact process survives in certain (tilted) space-time slabs  together with the monotonicity properties of  $\rho(\cdot)$. Section \ref{sec contact Markus} is dedicated to the  proof of Theorem \ref{prop contact process}.

\subsection{Discussion}\label{sec disc}
\begin{enumerate}
\item The SLLN for random walks on a $2$-state IPS has  been proven earlier by \citet*{AvenaHollanderRedigRWDRELLN2011}  under strong mixing assumptions on the environment, known as cone mixing. This has been extended to more general IPS by \citet{RedigVolleringRWDRE2013}, however, still under a uniform mixing assumption similar to cone mixing.

Theorem \ref{thm main} in this paper yields an extension of the SLLN in \cite{AvenaHollanderRedigRWDRELLN2011}  to random walks on IPS which are not cone mixing, but satisfy a monotonicity property. Indeed, instead of cone mixing, we assume that the IPS has an attractive graphical representation coupling and is started from a configuration where all sites are equal. Theorem \ref{prop coupling speed} present sufficient mixing conditions for relaxing the restriction on the starting configuration.  

Contrary to \cite{AvenaHollanderRedigRWDRELLN2011} and \cite{RedigVolleringRWDRE2013}, it is essential to the proof of Theorem \ref{thm main} that the random walk only has two transition kernels. That is, at jump times, the random walk chooses one among two transition kernels, depending on the environment. It is not clear how to extend our argument to random walks having more than two  transition kernels. In fact, there are examples showing that the monotonicity property crucial to our proof (see Lemma \ref{lem mono}) does not always hold for such systems already when the random walk depends on three states; see \citet{HolmesSalisburyRWDRE2012}.

\item Important to the proof of Theorem \ref{thm main} is the fact that a certain functional of the environment $\xi$ and $(W_t)$ is monotone in $\xi$. This functional counts, as a function of time, the number of occupied sites the random walk observes at the jump times of the random walk (see Section \ref{sec construction} for a definition). The monotonicity property of this functional is proven in Lemma \ref{lem mono}. We note that this property has earlier been exploited by \citet{HolmesSalisburyRWDRE2012} to study monotonicity properties of random walks on i.i.d.\ \emph{static} $2$-state random environments.

In several recent works on nearest neighbour RWDRE on $\integers$, e.g.\ \citet*{HilarioHollanderSidoraviciusSantosTeixeiraRWRW2014} and \citet{HuveneersSimenhausRWSEP2014}, monotonicity properties of the random walk have played an important role. Lemma \ref{lem mono} seems useful in order to extend their results to random walks on $\integers^d$ with more general transition kernels.

\item In \citet*{PeresPopovSousiRTRW2013}, sufficient conditions for general RWRE models to be transient were proven. In particular, \cite[Proposition 1.4]{PeresPopovSousiRTRW2013}  implies that $(W_t)$, as studied in this paper, is transient if it is elliptic and $d \geq 5$. If $u_1=u_0=0$, and under weak moment assumptions, (\cite[Theorem 1.2]{PeresPopovSousiRTRW2013}) yields that $(W_t)$ is transient when $d \geq 3$.

\item Theorem \ref{prop coupling speed} can be extended to hold for measures different from $\bar{\nu}_0$ and $\bar{\nu}_1$, see the remark at the end of Section \ref{sec proofs coupling}. In particular, the statement of Theorem \ref{prop coupling speed} holds if $(W_t)$ is elliptic and the dynamic environment is initialised from a measure $\mu \in \mathcal{P}(\Omega)$ which satisfies \eqref{eq conemixing} (with $\mu$ replacing $\bar{\nu}_i$). As an example, the contact process started from any measure $\mu$ stochastically dominating a non-trivial Bernoulli product measure satisfies \eqref{eq conemixing} with $i=1$, as follows from the proof of Theorem \ref{prop contact process}a).

\item Large deviation estimates, such as those obtained in Theorem \ref{thm ldp estimate W}, have previously been obtained in \citet*{AvenaHollanderRedigRWDRELDP2009} and \citet{RedigVolleringLTRWDRE2011}.
In \cite{RedigVolleringLTRWDRE2011}, explicit estimates are derived, but under strong mixing assumptions.
Closest to this paper is \cite{AvenaHollanderRedigRWDRELDP2009}, where an annealed as well as a quenched large deviation  principle is derived, however, restricted to nearest neighbour random walks on attractive spin-flip dynamics on $\integers$.
Note that, Theorem \ref{thm ldp estimate W} extends the estimates in  \cite[Proposition 2.5]{AvenaHollanderRedigRWDRELDP2009} to hold throughout the cone mixing regime and for random walks on $\mathbb{Z}^d$ with $d\geq1$.

\item The proof of Theorem \ref{prop contact process}b) can be adapted to more general dynamics 
such as  general additive and attractive spin-flip systems in the supercritical regime. At least when $d\geq 2$, our proof  seems to transfer to  this case using that such processes also survives in (tilted) space-time slabs, as shown in \citet{BezuidenhoutGrayCASS1994}.

Non-triviality of the speed for RWDRE as in Theorem \ref{prop contact process}b) has previously been proven by dos Santos \cite{SantosRWSEP2013} for a random walk on the exclusion process, by employing multi-scale arguments. This argument can perhaps be adapted to yield a different proof of $\rho(\lambda) \in (0,1)$ for the random walk on the contact process.
\end{enumerate}

\section{Construction}\label{sec construct}
 
\subsection{Coupling construction of the random walk}\label{sec construction}
In this section, we describe a particular coupling construction of the random walk. This construction is at the heart of the argument for the proof of Theorem \ref{thm main}, as it yields an important monotonicity property; see Lemma \ref{lem mono} below.

To construct the evolution of the random  walk, let $(N_t)$ be a Poisson jump process with jump rate $\gamma \in (0,\infty)$ and with inverse process $(J_k)_{k \geq 0}$. We call these times the jump times of the random walk. 
Essential to our approach and for the proof of Theorem \ref{thm main} is the introduction of two independent sequences of i.i.d.\ {\sc unif}[0,1] random variables, $O=(O_j)_{j \geq 1}$ and $V= (V_j)_{j\geq 1}$. Here $O$ stands for occupied, whereas $V$ stands for vacant.

Given $\alpha = (\alpha(i,z))_{i\in \set{0,1}, z \in \mathbb{Z}^d}$ as introduced in \eqref{eq jump rate walker}, enumerate $\mathbb{Z}^d = \set{z_1,z_2,\dots}$ and define \newline
$(p_1(m))_{m=0}^{\infty}$ and $(p_0(n))_{n=0}^{\infty}$ by setting $p_1(0)=p_0(0)=0$ and
\begin{align}
p_1(m)=\frac1\gamma\sum_{j=1}^m \alpha(1,z_j)\quad  \text{ and } \quad p_0(n)=\frac1\gamma\sum_{j=1}^n \alpha(0,z_j),\quad n,m\in \nat. \end{align} 
For convenience, we shall assume that 
the maximum in \eqref{eq jump rate walker} is attained by both $\alpha(1,\cdot)$ and $\alpha(0,\cdot)$  
by adapting the values for $\alpha(0,o)$ and $\alpha(1,o)$ appropriately. Note that this does not affect the behaviour of the random walk. Moreover, we can arrange that 
\begin{align}
\lim_{m\to\infty}p_1(m)=\lim_{n\to\infty}p_0(n)=1.
\end{align} 
Given a fixed environment $(\xi_t)_{t\geq0}$ we next define the \emph{discrete-time}
 random walk $S=(S_k)_{k\in\nat}$. For this, we also introduce the functional $(\rho(k,\xi))_{k \in \nat}$, taking values in $\nat$.
 Let $S_0 = o$ and $\rho(0,\xi)=0$ and, given $S_k$ and $\rho(k,\xi)$, define  $S_{k+1}$ and $\rho(k+1)$ iteratively by
\begin{equation}
\rho(k+1,\xi) = \rho(k,\xi) + \xi_{J_k}(S_k)
\end{equation} and
\begin{equation}\label{eq walker construction}
\begin{aligned}
S_{k+1} = S_k &+ (1-\xi_{J_k}(S_k)) \sum_{n=1}^{\infty} 1_{[p_0(n-1),p_0(n))}(V_{k+1-\rho(k,\xi)}) z_n \\&+ \xi_{J_k}(S_k) \sum_{m=1}^{\infty} 1_{[p_1(m-1),p_1(m))}(O_{\rho(k,\xi)+1}) z_m.
\end{aligned}
\end{equation}
Note that, in \eqref{eq walker construction}, since  $(\xi_t)$ is c\`adl\`ag and independent of $(J_k)_{k\geq0}$, we have that $\xi_{J_{k}^-}(S_k)$ is a.s.\ equal to $\xi_{J_k}(S_k)$ for all $k$. 
Further, we have that 
\begin{equation}
\rho(k,\xi) = \sum_{i=0}^{k-1} \xi_{J_i}(S_i), \quad k\in\nat, 
\end{equation} 
counts the number of occupied sites the (discrete-time) random walk $(S_k)$ has observed at the first $k$ jump times.
The (continuous-time) random walk $(W_t)$ with transition kernels $\alpha(i,\cdot)$, $i \in \{0,1\}$, is obtained by setting $W_t = S_{N_t}$, as follows by the construction, using that $(\xi_t)$ is right-continuous. The continuous version of the (rescaled) $\rho(n,\xi)$ is given by 
\begin{equation}\label{def density}
\rho_t(\xi) := \frac1{\gamma} \rho(N_t,\xi), \quad t\in [0,\infty).
\end{equation}

\subsection{Generalisation of the coupling construction}\label{sec construct general}

In the proof of Theorem \ref{prop contact process}b) in Section \ref{sec contact a} we carry out a domination argument. 
For this purpose, we consider a generalisation of the construction in Section \ref{sec construction}. 
For any $t \in [0,\infty)$, denote by $\xi_{[0,t]}$ the space-time environment from time $0$ to time $t$. 
Furthermore, consider a family of Boolean functions $(f_k)$, measurable 
 with respect to the $\sigma$-algebra $\sigma(\xi_{[0,J_k]},N,O,V)$. 
The construction in Section \ref{sec construction} can then be generalised in the same manner by setting $S_0=0$ and $\rho^{(d)}(0,\xi)=0$, and iteratively,
\begin{align}\label{eq general construction}
\begin{split}
&\rho^{(d)}(k+1,\xi) = f_{k+1} + \rho^{(d)}(k,\xi) 
\\ &S_{k+1} = S_k + (1-f_{k}) \sum_{n=1}^{\infty} 1_{[p_0(n-1),p_0(n))}(V_{k+1-\rho^{(d)}(k,\xi)}) z_n
\\& \quad \quad \quad \quad  \quad + f_k \sum_{m=1}^{\infty} 1_{[p_1(m-1),p_1(m))}(O_{\rho^{(d)}(k,\xi)+1}) z_m.
\end{split}
\end{align}
Thus, in this more general setup, $\rho^{(d)}(k,\xi) = \sum_{i=0}^{k-1} f_i$, $ k\in\nat$.
Note that, we recover \eqref{eq walker construction} when $f_k = \xi_{J_k}(S_k)$.
In Section \ref{sec contact a}, we consider cases where $f_k= \ind_{R_k} \xi_{J_k}(S_k)$ for some event $R_k \in \sigma(\xi_{[0,J_k]},N,O,V)$, for which we readily see that $\rho^{(d)}(k,\xi) \leq \rho(k,\xi)$.

\subsection{Monotonicity}\label{sec mono}
The construction in the previous two subsections provides us with a coupling that keeps track of the number of occupied sites the random walk has observed at any given time. The key property of the coupling construction in Subsection  \ref{sec construction} is a monotonicity property, which we state next. For this, denote the elements of $D_{\Omega}[0,\infty)$ by $(\eta_t)_{t\geq0}$ and write $(\eta_t)_{t\geq0} \leq (\omega_t)_{t\geq0}$ if $\eta_t(x) \leq \omega_t(x)$ for all $x \in \mathbb{Z}^d$ and $t \in [0,\infty)$.
\begin{lem}[Monotonicity of particle density]\label{lem mono}
For any $\eta=(\eta_t)_{t\geq0}$ and $\omega= (\omega_t)_{t\geq0}$ contained in $D_{\Omega}[0,\infty)$ satisfying  $(\eta_t)_{t\geq0} \leq (\omega_t)_{t\geq0}$;
\begin{equation}\label{eq lem mono supermann}
\rho_{s}(\eta) \leq \rho_{s}(\omega) \quad \forall \: s \in [0,\infty).
\end{equation}
 \end{lem}

 \begin{proof}
  Consider two (discrete-time) random walkers $(S_n^{1})_{n \geq 0}$ and $(S_m^{2})_{m \geq 0}$, both constructed as in Subsection \ref{sec construction} using the same realisation of $((N_t),(O_k),(V_k))$ and having identical  transition kernels, seeing environment $(\eta_t)$ and $(\omega_t)$, respectively. We claim that 
  \begin{equation}\label{eq lem mono supertarzan}
  \rho (n,\eta)\leq \rho(n,\omega) \quad \text{ for all } n \geq0.
  \end{equation}
  To see this, we argue by induction. First note that, by definition, we have that $\rho(0,\eta)=\rho(0,\omega)$. As the induction hypothesis, we assume that $\rho(k,\eta)\leq \rho(k,\omega)$ for some $k \geq 0$.
  
  In the case that $\rho(k,\eta)< \rho(k,\omega)$, we have  $\rho(k+1,\eta)\leq \rho(k,\eta)+1 \leq \rho(k,\omega) \leq \rho(k+1,\omega)$. 
  Thus, the induction step holds in this case.
  In the other case, where  $\rho(k,\eta)=  \rho(k,\omega)$, we have by construction that 
 \begin{align}
 S_k^{1} &= \sum_{i=1}^{\rho(k,\eta)} \big[ \sum_{m=1}^{\infty} 1_{[p_1(m-1),p_1(m))}(O_i) z_m \big]
 \\ &+ \sum_{j=1}^{k-\rho(k,\omega)} \big[ \sum_{n=1}^{\infty} 1_{[p_0(n-1),p_0(n))}(V_j) z_n \big] = S_k^{2}.\end{align} 
Since $\eta \leq \omega$, it holds that 
$\eta_{J_k}(S_k^1) \leq \omega_{J_k}(S_k^2)$, which in particular implies that $\rho(k+1,\eta) \leq \rho(k+1,\omega)$. Hence, also the second case satisfy the induction step and consequently \eqref{eq lem mono supertarzan} holds. Finally, by replacing $n$ in  \eqref{eq lem mono supertarzan} by $N_t$ and multiplying by $\gamma^{-1}$, this proves \eqref{eq lem mono supermann}.
\end{proof}

\begin{rem}
For any attractive IPS $\xi$, Lemma \ref{lem mono} transfers to an almost sure statement with respect to the annealed measure (and thus also the quenched measure). In this case, for every  $\mu ,\nu \in \mathcal{P}(\Omega)$ with $\mu\leq \nu$ there exists a coupling $\widehat{\Prob}$ of $\Prob_{\mu,o}$, $\Prob_{\nu,o}$ and $\xi$ such that $\widehat{\Prob}( \rho_{t}(\xi^{\mu}) \leq \rho_{t}(\xi^{\nu})) =1$. The existence of such a coupling follows by \cite[Theorem II.2.4]{LiggettIPS1985} and the construction above.
\end{rem}

\begin{rem}
The coupling construction in Section \ref{sec construction} for random walks on a \emph{dynamic} random environment is to our knowledge new. Apparently the same coupling construction has previously been used to study monotonicity properties for certain specific random walks in \emph{static} random environment by \citet{HolmesSalisburyRWDRE2012}.  Lemma \ref{lem mono} can be seen as an immediate extension of \cite[Theorem 4.1i)]{HolmesSalisburyRWDRE2012}. 
\end{rem}

\begin{rem}
Lemma \ref{lem mono} can be extended to hold in certain cases under the general construction considered in Section \ref{sec construct general}. For this, the functions $(f_k)$ need to be monotone in the sense that  (for $\eta \leq \xi$ as above), if $\rho^{(d)}(k,\eta) = \rho^{(d)}(k,\xi)$, then $\rho^{(d)}(k+1,\eta)\leq \rho^{(d)}(k+1,\xi)$.
\end{rem}

\section{Proofs}\label{sec proofs Markus}

\subsection{Proof of Theorem \ref{thm main}}\label{sec proof lln}

In this subsection we first present the proof of Theorem \ref{thm main} for the case when the environment is started from all sites occupied. 
Essentially the same proof can be applied to the case where the environment is started from all sites vacant. We comment at the end of this subsection on which changes to the proof are necessary for this case.

The main idea is to show that $\rho_t(\xi)$ is sub-additive, by using that $\xi$ has an attractive graphical representation coupling and Lemma \ref{lem mono}. 
Subsequently, the subadditive ergodic theorem applied to $\rho_{t}(\xi)$ yields that $t^{-1} \rho_t(\xi)$ converges towards a deterministic constant. This, in turn, identifies the limiting speed. 

Let $\xi$ be an IPS with an attractive graphical representation coupling, $\widehat{P}$, where  by $\mathcal{I}$ we denote the corresponding collection of Poisson point processes. 
In order to formulate the proof, we have to be more specific about $\mathcal I$ and write $\mathcal I$ as a countable set of Poisson point processes indexed by the lattice $\mathbb{Z}^d$, $\mathcal I=\left((\mathcal{X}_y^1)_{y\in\mathbb{Z}^d}, (\mathcal{X}_y^2)_{y\in\mathbb{Z}^d},\dots\right)$, where every $\mathcal{X}^i_y$ is an (independent) Poisson point process on $[0,\infty)$.  Further, for $x\in\mathbb{Z}^d$ and $t\in[0,\infty)$ let $\Theta_{x,t}$ be the space-time shift operator on the realisations of $\mathcal{I}$: 
$$ \Theta_{x,t}\left((\mathcal{X}_{s,y}^1), (\mathcal{X}_{s,y}^2),\dots\right)_{s\in[0,\infty),y\in\mathbb{Z}^d}
=\left((\mathcal{X}_{s+t,y+x}^1), (\mathcal{X}_{s+t,y+x}^2),\dots\right)_{s\in[0,\infty),y\in\mathbb{Z}^d}.$$
For the contact process, the set $\mathcal I$ as considered in Section \ref{sec contact Markus}, consists of the Poisson processes   $\left\{H_x,  I_{x,e} \colon x,e\in\mathbb{Z}^d,|e|=1\right\}$, and $\Theta_{x,t}$ shifts crosses and arrows in space by $x$ and in time by $t$. Further, in \cite[Theorem 2.5]{DurrettIPS1995}, $\mathcal{I}$ is the set of birth and death events (which in \cite{DurrettIPS1995} are denoted  by 
 $\left\{T_n^{y,i} \colon  n\geq0,y \in \mathbb{Z}^d, i \in \{0,1\} \right\}$ ).

To emphasise the graphical representation, we write $\rho_{t}(\xi) = \rho_{t}(\eta,\mathcal{I},N,O,V)$ for $\xi_0=\eta$ and let $\widehat{\Prob}$ denote the joint law of the graphical construction coupling and $N,O$ and $V$.
Note that, by Lemma \ref{lem mono} and \eqref{eq agr}, for any $\eta \in \Omega$,
\begin{equation}\label{eq mono 3.1}\rho_{t}(\eta,\mathcal{I},N,O,V) \leq \rho_{t}(\bar{1},\mathcal{I},N,O,V), \quad \widehat{\Prob}-a.s.\end{equation}
Moreover, let
\begin{align*}
&N^{(s)} = (N_t^{(s)})_{t \geq 0} := (N_{t+s}-N_s)_{t \geq 0},
\\&O^{(s)} = (O_n^{(s)})_{n \geq 0} := (O_{n + \gamma \rho_s(\xi)})_{n \geq 0},
\\&V^{(s)} = (V_n^{(s)})_{n \geq 0} := (V_{n+ N_s-\gamma \rho_s(\xi)})_{n \geq 0}.
\end{align*}
Similar to $\Theta_{x,t}$ we introduce the space-time shift $\theta_{x,t}$  on $\Omega^{[0,\infty)}$ by 
\[ (\theta_{x,t}\xi_s)_{s\ge0} = (\theta_{x}\xi_{s+t})_{s \geq 0} \]
with space-shift $\theta_x$ introduced in \eqref{eq stationarity}. 
Next, define the continuous-time process $(X_{t,s})_{0\leq t \leq s}$ by
\begin{equation}\label{eq sub pro}
X_{t,s} := \rho_{s-t}(\bar{1},\Theta_{W_t,t} \mathcal{I}, N^{(t)},O^{(t)},V^{(t)}), \quad  \text{ for } 0 \leq t \leq s.
\end{equation}
Note that, if $\xi$ is such that $\xi_0=\bar1$, then
\begin{equation}\label{X eta equiv}
	X_{0,s}=\rho_s(\bar 1,\mathcal{I},N,O,V) = \rho_s(\xi),
	\qquad  s\in [0,\infty). 
\end{equation}
In the next statement and in the proceedings, for $\mu \in \mathcal{P}(\Omega)$ and $x \in \mathbb{Z}^d$, we write  $\widehat{\Prob}_{\mu,x}$ to  emphasise the starting configuration of both $\xi$ and $(W_t)$.
\begin{lem}[Sub-additivity]\label{lem subadd}
 The process $(X_{t,s})_{0\leq t \leq s}$ has the following properties.
\begin{description}
\item[i)] $X_{0,0}=0$ and for all $t,s \in [0,\infty)$: $X_{0,t+s}\leq X_{0,t} + X_{t,t+s}$.
\item[ii)] For all $t \in (0,\infty)$, $(X_{t,k+t})_{k\geq 1}$ has the same distribution as  $(X_{0,k})_{k\geq 1}$.
\item[iii)] For all $t \in [0,\infty)$, $(X_{(k-1)t,kt})_{k\geq1}$ is a sequence of i.i.d.\ random variables.
\item[iv)] For all $t \in [0,\infty)$, the expectation $\widehat{\mathbb{E}}_{\delta_{\bar{1}},o}[X_{0,t}]$ is finite and $X_{0,t} \geq 0$.
\end{description}
\end{lem}

\begin{proof}
Fix $t,s\in [0,\infty)$ and recall \eqref{eq mono 3.1}. By the Markov property of the Poisson point process $\mathcal{I}$, we have that
\begin{align*}
X_{0,t+s} &= \rho_{t+s}(\bar{1},\mathcal{I},N,O,V) 
\\ &= \rho_{t}(\bar{1},\mathcal{I},N,V,O) + \rho(s, \theta_{W_t,t} \xi, \Theta_{W_t,t} \mathcal{I},N^{(t)},O^{(t)},V^{(t)})
\\ &\leq \rho_{t}(\bar{1}, \mathcal{I},N,O,V)+ \rho(s, \bar{1}, \Theta_{W_t,t}\mathcal{I},N^{(t)},O^{(t)},V^{(t)})
\\ &= X_{0,t} + X_{t,t+s}.
 \end{align*}
Properties i) and ii) follow from the equality $X_{0,0}=0$, the translation invariance in \eqref{eq stationarity} and the equality in distribution  $X_{0,s} = X_{t,t+s}$. 
Moreover, iii) follows by the Markov property of $\xi$ and the graphical representation.
Lastly, property iv) holds trivially, since $X_{0,t}$ is non-negative by definition and since $X_{0,t}\leq N_t$. \end{proof}

Lemma \ref{lem subadd} enables us to prove the SLLN for the process $\rho_{t}(\xi)$ when $\xi$ is initialised at time $0$ by $\bar{1}$, by applying the subadditive ergodic theorem.

\begin{thm}[Law of large numbers for $\rho_{t}(\xi)$]\label{cor main}
Assume that  $\xi$ has an attractive graphical representation coupling.
There exists $\rho_1 \in [0,1]$ such that
\begin{equation}\label{eq lln rho}
\lim_{t \rightarrow \infty} \frac{1}{t}\rho_{t}(\xi)  = \rho_1 \qquad \widehat{\Prob}_{\delta_{\bar{1}},o}\text{-a.s. and in } L^1.
\end{equation}
Moreover, $\rho_1 = \inf_{t \geq 1} t^{-1} \widehat{\mathbb{E}}_{\delta_{\bar{1}},o}(\rho_{t}(\xi))$.
\end{thm}

\begin{proof}
By Lemma \ref{lem subadd} we know that $X$ satisfies property a)-d) of \cite[Theorem VII.2.6]{LiggettIPS1985}. In particular, by the independence property in Lemma \ref{lem subadd}iii), the process is stationary and ergodic. Hence, the conclusion of Theorem \ref{cor main} holds  when $t$ takes integer values. This can easily be extended to continuous $t$. Indeed, for any $t\in (0,\infty)$ we have that 
\begin{align}\label{eq cor main help}
X_{0,\lfloor t \rfloor} \leq X_{0,t} \leq X_{0,\lceil t \rceil}.
\end{align}
In particular, by dividing by $t$ in \eqref{eq cor main help} and  
 taking $t \rightarrow \infty$ (as in \eqref{eq lln rho}), we conclude the proof.
 \end{proof}

We are now in position to present the proof of Theorem \ref{thm main}.

\begin{proof}[Proof of Theorem \ref{thm main}]
By the construction in Section \ref{sec construction}, $W_t$ can be written as
\begin{equation}
W_t = \sum_{i=1}^{\rho(N_t,\xi)}  \left(\sum_{m=1}^{\infty} 1_{[p_1(m-1),p_1(m))}(O_{i})z_m \right)+ \sum_{j=1}^{N_t-\rho(N_t,\xi)} \left(\sum_{n=1}^{\infty} 1_{[p_0(n-1),p_0(n))}(V_j)z_n\right).
\end{equation}
Dividing by $t >0$ gives 
\begin{align} \label{eq rho speed relation2}
\frac{W_t}{t}=
&\frac{ \rho(N_t,\xi)}{t} \frac{1}{\rho(N_t,\xi)} \sum_{i=1}^{\rho(N_t,\xi)}  \left(\sum_{m=1}^{\infty} 1_{[p_1(m-1),p_1(m))}(O_{i})z_m \right) \\
&+ \frac{N_t- \rho(N_t,\xi)}{t} \frac{1}{N_t - \rho(N_t,\xi)} \sum_{j=1}^{N_t- \rho(N_t,\xi)} \left(\sum_{n=1}^{\infty} 1_{[p_0(n-1),p_0(n))}(V_j)z_n\right).\end{align}
 Taking the limit as $t \rightarrow \infty$ and applying Theorem  \ref{cor main} we obtain 
\begin{equation}\label{eq rho speed relation}
\lim_{t \rightarrow \infty} \frac{1}{t}W_t = \rho_1 u_1 + (1-\rho_1) u_0 \quad \Prob_{\delta_{\bar{1}},o}\text{-a.s. and in } L^1,
\end{equation}
where 
$\rho_1$ is as in Theorem  \ref{cor main} and $u_0,u_1 \in \reals^d$ are as in \eqref{eq speed SRW}. This proves Theorem \ref{thm main} for the case when the environment is started from all sites equal to $1$. 

We next comment on the changes necessary in the argument for proving Theorem \ref{thm main} when started from all sites equal to $0$. For this case we can define the process $(Y_{t,s})_{0\leq t \leq s}$ given by
\begin{equation}\label{eq super process}
Y_{t,s} := \rho(s-t,\bar{0},\Theta_{W_t,t} \mathcal{I}, N^{(t)},O^{(t)},V^{(t)})  \text{ for } 0 \leq t \leq s.
\end{equation}
By the same arguments as in Lemma \ref{lem subadd} we can prove that $-Y$ is a sub-additive process satisfying property ii) and iii) as in Lemma \ref{lem subadd}. Moreover, since $Y_{0,t}$ is dominated by $N_t$ it follows that $\widehat{\mathbb{E}}_{\delta_{\bar{0}},o}[Y_{0,t}] \leq \widehat{\mathbb{E}}_{\delta_{\bar{0}},o}[N_t] = t$. This is sufficient in order to apply  \cite[Theorem VII.2.6]{LiggettIPS1985}. By a literal adaptation of the proof under $\widehat{\Prob}_{\delta_{\bar{1}},o}$ above we obtain
\begin{equation}
\lim_{t \rightarrow \infty} \frac{1}{t}W_t = \rho_0 u_1 + (1-\rho_0) u_0 \quad \Prob_{\delta_{\bar{0},o}}\text{-a.s. and in } L^1,
\end{equation}
where 
$\rho_0$ is the limit  in Theorem \ref{cor main} when $\widehat{\Prob}_{\delta_{\bar{1}},o}$ is replaced by $\widehat{\Prob}_{\delta_{\bar{0}},o}$. This completes the proof of Theorem \ref{thm main}.
\end{proof}

\subsection{Proof of Theorem \ref{prop coupling speed}}\label{sec proofs coupling}
In this subsection we present  the proof of Theorem \ref{prop coupling speed}. 
The presentation is inspired by the proof of \cite[Proposition 3.3]{HollanderSantosRWCP2013} (see also Remark 3.4 therein), and the proof of 
Theorem \ref{prop coupling speed} is an extension of their proof to higher dimensions. We only provide the proof when $i=1$. The proof for the case $i=0$ is analogous. 

\begin{proof}[Proof of Theorem \ref{prop coupling speed}]
We start with the construction of the random walk. 
Let $U:= (U_k)_{k\in \nat_0}$ be an i.i.d.\ sequence of {\sc unif}$[0,1]$ random variables, independent of the jump process $N=(N_t)_{t\geq0}$. Set $S_0^{(U)} :=0$ and, recursively for $k \geq 0$,
\begin{align*}
S_{k+1}^{(U)} := S_k^{(U)} &+ (1-\xi_{J_k}(S_k^{(U)})) \sum_{n=1}^{\infty} 1_{[p_0(n-1),p_0(n))}(U_k) z_n \\&+ \xi_{J_k}(S_k^{(U)})\sum_{m=1}^{\infty} 1_{[p_1(m-1),p_1(m))}(U_k) z_m,
\end{align*}
and let $W_t^{(U)} := S_{N_t}^{(U)}$ and $\rho_t^{(U)}= \sum_{k=0}^{N_t} \xi_{J_k}(S_k)$. Clearly $W_t^{(U)}$ and $W_t$ (as constructed in Section \ref{sec construction}) are equal in distribution, and similarly for $\rho_t^{(U)}$ and $\rho_t$, and hence onwards we do not distinguish them and write $W_t$ and $\rho_t$ for both processes.

Let
$N^{(\bar{\nu}_1)}$,
$U^{(\bar{\nu}_1)}$ and 
$N^{(\delta_{\bar{1}})}$,
$U^{(\delta_{\bar{1}})}$ be independent copies of $U$ and $N$ and denote by $\widehat{\Prob}$ the joint law of $\widehat{P},N^{(a)}, U^{(a)}$, $a \in \set{\bar{\nu}_1,\delta_{\bar{1}}}$. Then $W^{(\bar{\nu}_1)}:= W(\xi^{(\bar{\nu}_1)},N^{(\bar{\nu}_1)},U^{(\bar{\nu}_1)})$  and $\rho^{(\bar{\nu}_1)}:= \rho(\xi^{(\bar{\nu}_1)},N^{(\bar{\nu}_1)},U^{(\bar{\nu}_1)})$ under $\widehat{\Prob}$ have the same law as $W$ and $\rho$ under $\widehat{\Prob}_{\bar{\nu}_1,o}$. Similarly, $W^{(\delta_{\bar{1}})}:= W(\xi^{(\delta_{\bar{1}})},N^{(\delta_{\bar{1}})},U^{(\delta_{\bar{1}})})$  and $\rho^{(\delta_{\bar{1}})}:= \rho(\xi^{(\delta_{\bar{1}})},N^{(\delta_{\bar{1}})},U^{(\delta_{\bar{1}})})$  have the same law as $W$  and $\rho$ under $\Prob_{\delta_{\bar{1}},o}$.
Further, for $T>0$ fixed, let 
$\hat{N}=(\hat{N}_s)_{s\geq0}$ and
$\hat{U}=(\hat{U}_n)_{n\in \nat}$ be defined by
\begin{equation}
\hat{U}_n :=
\left\{
	\begin{array}{ll}
		U_n^{(\delta_{\bar{1}})}  & \mbox{if } n \leq N_T^{(\delta_{\bar{1}})}; \\
		U_n^{(\bar{\nu}_1)} & \mbox{otherwise, } 
	\end{array}
\right.
\end{equation}
\begin{equation}
\hat{N}_s :=
\left\{
	\begin{array}{ll}
		N_s^{(\delta_{\bar{1}})}  & \mbox{if } s \leq T; \\
		N_T^{(\delta_{\bar{1}})} + N_s^{(\bar{\nu}_1)} -N_T^{(\bar{\nu}_1)} & \mbox{otherwise. } 
	\end{array}
\right.
\end{equation}
It is clear that $\hat{W} := W(\xi^{(\delta_{\bar{1}})},\hat{N},\hat{U})$ and $\hat{\rho} := \rho(\xi^{(\delta_{\bar{1}})},\hat{N},\hat{U})$ have the same laws as $W^{(\delta_{\bar{1}})}$ and $\rho^{(\delta_{\bar{1}})}$. Furthermore, $\hat{N}$ and $N^{(\bar{\nu}_1)}$ are independent up to time $T$, and thus the jump times of $W^{(\bar{\nu}_1)}$ and $\hat{W}$ are independent in the time interval $[0,T]$. By construction, for times later than $T$, the jumping times of $W^{(\bar{\nu}_1)}$ and $\hat{W}$ are the same.

Next, let $\epsilon, m > 0$ be such that $U(u_0,u_1) \times \{1/\gamma\} \subset V_{m(1-\epsilon)}$ and 
\begin{equation}\label{eq  proof conemixing help}
\lim_{T \rightarrow \infty} \phi_1^{(m)}(\eta,T)=0, \quad \text{ for } \bar{\nu}_1 \text{-a.e. } \eta \in \Omega. 
\end{equation}
For $T\geq 0$, define the event
\begin{equation}
D_T := \set{ (W_t^{(\bar{\nu}_1)},t) \in V_{m(1-\epsilon)}, \: \forall \: t \geq T }
\end{equation}
and let
\begin{equation}
\Gamma_T := D_T \cap \set{ \xi_s^{(\bar{\nu}_1)}(x) = \xi_s^{(\delta_{\bar{1}})}(x), \: \forall \: (x,s) \in V_m \cap \mathbb{Z}^d \times [T,\infty)}.
\end{equation}
Note that, since $\widehat{\Prob}(\Gamma) \geq 1 - \left( \widehat{\Prob}(D_T^c) + \int \phi_1^{(m)}(\eta,T) \bar{\nu}_1(d\eta) \right)$, by  \eqref{eq  proof conemixing help} and since $U(u_0,u_1) \times \{1/\gamma\} \subset V_{m(1-\epsilon)}$, it holds that 
$\lim_{T \rightarrow \infty} \widehat{\Prob}(\Gamma_T)=1.$ 
Furthermore, by stationarity under $\bar{\nu}_1$, and since $\hat{N}_T$ is independent of  $(W_t^{(\bar{\nu}_1)})$ and $(\xi_t)$, we have that
\begin{align*}
\widehat{\Prob} \left(\lim_{t \rightarrow \infty} t^{-1} \rho_t^{(\bar{\nu}_1)} = \rho_1 \right)  
&= \widehat{\Prob} \left(\lim_{t \rightarrow \infty} t^{-1} \rho_t^{(\bar{\nu}_1)} = \rho_1 \mid N_T^{(\bar{\nu}_1)}= \hat{N}_T=0 \right) 
\\ &\geq \widehat{\Prob} \left( \left\{ \lim_{t \rightarrow \infty} t^{-1} \rho_t^{(\bar{\nu}_1)} = \rho_1\right\} \cap \Gamma_T \mid N_T^{(\bar{\nu}_1)}= \hat{N}_T=0 \right) 
\\ &=  \widehat{\Prob} \left( \left\{ \lim_{t \rightarrow \infty} t^{-1} \rho_t^{(\delta_{\bar{1}})} = \rho_1\right\} \cap \Gamma_T \mid N_T^{(\bar{\nu}_1)}= \hat{N}_T=0 \right) 
\\ & =  \widehat{\Prob} \left( \Gamma_T \mid  N_T^{(\bar{\nu}_1)}= \hat{N}_T=0 \right) =  \widehat{\Prob} \left( \Gamma_T \right) \xrightarrow{T \rightarrow \infty} 1.
\end{align*}
Hence $\widehat{\Prob} \left(\lim_{t \rightarrow \infty} t^{-1} \rho_t^{(\bar{\nu}_1)} = \rho_1 \right)   =1$. From this and the monotonicity property obtained in Lemma \ref{lem mono}, and by arguing as in the proof of Theorem \ref{thm main}, we conclude the proof. 
\end{proof}

\begin{rem}
The last paragraph of the proof of Theorem \ref{prop coupling speed} is based on \cite[Remark 3.4]{HollanderSantosRWCP2013}. Following the proof of \cite[Proposition 3.3]{HollanderSantosRWCP2013}, assuming that $(W_t)$ is elliptic, this part can be extended to hold for any measure $\mu\in \mathcal{P}(\Omega)$ for which \eqref{eq conemixing} holds, with $\bar{\nu}_1$ replaced by $\mu$. For this, since $W_T^{(\mu)} \in [-mT,mT]^d$ on $\Gamma_T$, we have that
\begin{align}
 &\widehat{\Prob} \left(\lim_{t \rightarrow \infty} t^{-1} W_t^{(\mu)} = v_1  \mid \Gamma_T\right)
\\= \sum_{x \in [-mT,mT]^d} &\widehat{\Prob}\left( \lim_{t \rightarrow \infty} t^{-1} W_t^{(\mu)} = v_1\mid \Gamma_T, W_T^{(\mu)}=x \right)  \widehat{\Prob} \left(  W_T^{(\mu)}=x \mid \Gamma_T \right).
 \end{align}
In particular, it is sufficient to show that
\begin{align}\label{rem prop coupling speed}
\widehat{\Prob} \left( \lim_{t \rightarrow \infty} t^{-1} W_t^{(\mu)} = v_1\mid \Gamma_T, W_T^{(\mu)}=x \right) =1, \: \forall \: x \in [-mT,mT]^d.
\end{align}
To prove the latter equation, use the ellipticity assumption to construct events $A_x$,  $x \in [-mT,mT]^d$, having the following properties:  $A_x \subset \{ \hat{W}_T = x \}$ and $A_x$ is independent of $(\xi_t)$ and $(W_t^{(\mu)})$. 
 Conclude \eqref{rem prop coupling speed} by first conditioning on $A_x$ and then noting that, under $\Gamma_T \cap \{W_T^{(\mu)}=x\} \cap A_x$, we can replace $\{\lim_{t \rightarrow \infty} t^{-1} W_t^{(\mu)}\}$ by $\{\lim_{t \rightarrow \infty} t^{-1}  \hat{W}_t\}$. 
\end{rem}

\subsection{Large deviations properties and proof of Theorem \ref{thm ldp estimate W}}\label{sec proofs ldp}
We constructed in Lemma \ref{lem subadd} an independent sub-additive process $(X_{t,s})_{0\leq t \leq s}$. Such processes are well known to satisfy large deviation properties, see e.g.\
\citet{GrimmettLDPSAP1985}. 
In particular, we have the following large deviation estimates for $\rho_t(\xi)$. 

\begin{thm}[Large deviation estimates for $\rho_{t}(\xi)$]\label{thm ldp estimate}
Assume that $\xi$ has an attractive graphical representation coupling. 
Then, for any $\epsilon>0$, there exists $R_i(\epsilon)>0$, $i \in \{0,1\}$, such that
\begin{align}
\begin{split}
&\widehat{\Prob}_{\delta_{\bar{1}},o} \left( \rho_{t}(\xi) > t( \rho_1 + \epsilon) \right) \leq \exp \left(-t R_1(\epsilon)\right),  \quad\text{ for all } t >0;
\\&\widehat{\Prob}_{\delta_{\bar{0}},o} \left( \rho_{t}(\xi) < t(\rho_0 - \epsilon) \right) \leq \exp \left(-t R_0(\epsilon) \right), \quad\text{ for all } t >0.
\end{split}
\end{align}
\end{thm}

\begin{proof}[Proof of Theorem \ref{thm ldp estimate}]
We follow the proof of Theorem 3.2 in \cite{GrimmettKestenFPP1984} and give the proof with respect to $\delta_{\bar{1}}$ only. The proof with respect to $\delta_{\bar{0}}$ follows analogously.
Let $\epsilon>0$ and choose $T>1$ such that
\begin{align}\label{eq ldp help 1.1}
g_T := \frac{1}{T} \widehat{\mathbb{E}}_{\delta_{\bar{1}},o} \big[X_{0,T} \big] \leq \rho_1 + \epsilon.
\end{align}

We first consider the case when $t=rT$ for some $r \in \nat$. Using the properties from Lemma \ref{lem subadd}, we have
\[ \widehat{\Prob}_{\delta_{\bar{1}},o} \left( X_{0,t} \geq t(\rho_1+2\epsilon)\right) \leq \widehat{\Prob}_{\delta_{\bar{1}},o} \left(Q_1 + \dots + Q_r \geq t(\rho_1 + 2 \epsilon) \right), \]
where $Q_i = X_{(i-1)T,iT}$. Moreover, the $Q_i$'s are i.i.d., and, since $Q_1$ is dominated by a Poisson random variable, $\widehat{\mathbb{E}}_{\delta_{\bar{1}},o}\big[e^{z Q_1}\big] < \infty$ for all $z \in \reals$. 

Next, let $Z_i = Q_i - \widehat{\mathbb{E}}_{\delta_{\bar{1}},o}(Q_i)$, and note that (by \eqref{eq ldp help 1.1})\begin{align*}
 \widehat{\Prob}_{\delta_{\bar{1}},o} \left( Q_1 + \dots + Q_r \geq  t (\rho_1 + 2 \epsilon) \right)
\leq \widehat{\Prob}_{\delta_{\bar{1}},o} \left( Z_1 + \dots + Z_r \geq  rT  \epsilon \right).
\end{align*}
Further, applying the exponential Chebyshev inequality implies that for each $y\geq0$,
\[ \widehat{\Prob}_{\delta_{\bar{1}},o} \left( Z_1 + \dots + Z_r \geq  rT \epsilon\right) \leq e^{-rT\epsilon y} \widehat{\mathbb{E}}_{\delta_{\bar{1}},o}\big[e^{yZ_1}\big]^r. \]
Since $\widehat{\mathbb{E}}_{\delta_{\bar{1}},o}\big[e^{yZ_1}\big]< \infty$ for all $y \leq 1$ and $\widehat{\mathbb{E}}_{\delta_{\bar{1}},o}\big[Z_1\big]=0$, there exists a constant $c= c(1) >0$ such that $\widehat{\mathbb{E}}_{\delta_{\bar{1}},o}\left[e^{yZ_1}\right] \leq 1+cy^2$  for  $y \in [0,1]$. 
Hence, by setting $y= \frac{\epsilon}{2c}$, for $r$ large,
\begin{align*}
\widehat{\Prob}_{\delta_{\bar{1}},o} \left( Z_1 + \dots + Z_r \geq  rT \epsilon \right) &\leq \exp \big[ -rT\epsilon y + r \log (1+cy^2) \big]
\\ &\leq \exp \big[-r T \epsilon y + rcy^2 \big]
\\ &= \exp \big[ -t\frac{  \epsilon^2}{4Tc}  \big].
\end{align*}
This completes the proof for the case when $t$ is a multiple of $T$.

For general values of $t$, write $t= rT + s$, where $0 \leq s < T$, and note that
\[ X_{0,t} \leq X_{0,rT} + X_{rT,t}, \]
where the last two variables are independent. Further, notice that we can bound
\[ \widehat{\Prob}_{\delta_{\bar{1}},o} \left( X_{0,t} \geq t(\rho_1+\epsilon) \right) \leq \widehat{\Prob}_{\delta_{\bar{1}},o} \left(X_{0,rT} \geq t(\rho_1 + \epsilon/2) \right) + \widehat{\Prob}_{\delta_{\bar{1}},o} \left( X_{0,s} \geq t \epsilon/2 \right).\]
By using that $X_{0,s}\leq \sum_{k=1}^{\lceil s \rceil} X_{(k-1),k}$ and Markov's inequality, we obtain that
\[ \widehat{\Prob}_{\delta_{\bar{1}},o} (X_{0,s} \geq t \epsilon) \leq e^{-t\epsilon} \widehat{\mathbb{E}}_{\delta_{\bar{1}},o}(e^{X_{0,1}})^{\lceil s \rceil}, \]
which completes the proof since $ \widehat{\mathbb{E}}_{\delta_{\bar{1}},o}(e^{X_{0,1}})^{\lceil s \rceil} \leq  \widehat{\mathbb{E}}_{\delta_{\bar{1}},o}(e^{X_{0,1}})^{\lceil T \rceil} < \infty$.
\end{proof}

We  continue with the proof of Theorem \ref{thm ldp estimate W}.

\begin{proof}[Proof of Theorem \ref{thm ldp estimate W}]
First note that, by the construction in Section \ref{sec construct}, we have that  
\begin{equation} W_t = \sum_{i=1}^{\rho(N_t,\xi)} \tilde{O}_i + \sum_{j=1}^{N_t-\rho(N_t,\xi)} \tilde{V}_j ,\end{equation}
where $\tilde{O}_i  = \sum_{m=1}^{\infty} 1_{[p_1(m-1),p_1(m))}(O_{i})z_m$ and $\tilde{V}_j =\sum_{n=1}^{\infty} 1_{[p_0(n-1),p_0(n))}(V_j)z_n$. 
Let $v=\rho u_1 + (1-\rho)u_0$, where $\rho:=\rho_1$ ($=\rho_0$ by assumption). Then, for $\epsilon>0$ and $\mu \in \mathcal{P}(\Omega)$, we have that
\begin{align*}
\Prob_{\mu,o} \left( \norm{W_t - t v}_1\geq t \epsilon \right)
 = &\widehat{\Prob}_{\mu,o} \left( \norm{ \sum_{i=1}^{\rho(N_t,\xi)} \tilde{O}_i  + \sum_{j=1}^{N_t-\rho(N_t,\xi)} \tilde{V}_j -tv}_1 \geq t \epsilon \right)
\\ \leq &\widehat{\Prob}_{\mu,o} \left( \norm{ \sum_{i=1}^{\rho(N_t,\xi)} \tilde{O}_i - t \rho u_1}_1 \geq t \frac{\epsilon}{2} \right) 
\\ + &\widehat{\Prob}_{\mu,o} \left( \norm{ \sum_{j=1}^{N_t- \rho(N_t,\xi)} \tilde{V}_j - t (1-\rho) u_0}_1 \geq t \frac{\epsilon}{2} \right).
\end{align*}
To conclude the proof it is thus sufficient to show that both the latter terms decay exponentially in $t$. Since our argument is almost identical for both terms, we only provide the detailed proof for the first one. By Theorem \ref{thm ldp estimate}, for any $\delta>0$,
\begin{align}
&\widehat{\Prob}_{\mu,o} \left( \norm{ \sum_{i=1}^{\rho(N_t,\xi)} \tilde{O}_i - t \rho u_1}_1 \geq t \frac{\epsilon}{2} \right)
\\ \label{eq help 3.14} \leq &\sum_{n=\lfloor t\gamma(\rho-\delta) \rfloor }^{\lceil t\gamma(\rho + \delta) \rceil} \widehat{\Prob}_{\mu,o} \left( \norm{ \sum_{i=1}^{n} \tilde{O}_i - t \rho u_1}_1 \geq t \frac{\epsilon}{2} \right) + \exp(-t R(\delta)),
\end{align}
where $R(\delta):=\min \{R_1(\gamma \delta),R_0(\gamma \delta)\}>0$. Indeed, the estimates in Theorem \ref{thm ldp estimate} applies to the process started from $\mu$, since by Lemma \ref{lem mono}, for any $\mu \in \mathcal{P}(\Omega)$ and $t, s>0$, we have 
$\widehat{\Prob}_{\delta_{\bar{0}},o} \left( \rho_t(\xi) \geq s \right) \leq \widehat{\Prob}_{\mu,o} \left(\rho_t(\xi) \geq s\right) \leq \widehat{\Prob}_{\delta_{\bar{1}},o} \left(\rho_t(\xi) \geq s\right).$
Further, for any integer $n \in t \gamma (\rho- \delta, \rho+\delta)$ we have that
\begin{align}\label{eq help 3.15}
\widehat{\Prob}_{\mu,o} \left( \norm{ \sum_{i=1}^{n } \tilde{O}_i - t \rho u_1}_1 \geq t \frac{\epsilon}{2} \right) \leq 
\widehat{\Prob}_{\mu,o} \left( \norm{ \sum_{i=1}^{n}  \tilde{O}_i - n  u_1}_1 \geq t (\frac{\epsilon}{2}-\norm{u}_1 \delta ) \right).
\end{align}
 Observe that, by taking $\delta>0$ small enough, we can guarantee that $\frac{\epsilon}{2}-\norm{u}_1 \delta>0$. In this case, since $(W_t)$ has finite exponential moments, \eqref{eq help 3.15} is exponentially small (in $n$, hence in $t$) by Cram\'ers Theorem applied to $(\tilde{O}_i)$. 
Applying this estimate to \eqref{eq help 3.14}, taking $\delta$ small, yields that for some $C_1,c_1>0$, \begin{align}
\widehat{\Prob}_{\mu,o} \left( \norm{ \sum_{i=1}^{\rho(N_t,\xi)} \tilde{O}_i - t \rho u_1}_1 \geq t \frac{\epsilon}{2} \right) \leq C_1e^{-c_1 t}.
\end{align}
 Noting that the same argument can be used to yield, for some $C_0,c_0>0$,
 \begin{align}
 \widehat{\Prob}_{\mu,o} \left( \norm{ \sum_{j=1}^{N_t-\rho(N_t,\xi)} \tilde{V}_j - t (1-\rho) u_0}_1 \geq t \frac{\epsilon}{2} \right) \leq C_0e^{-c_0t},
 \end{align} completes the proof of the theorem. 
\end{proof}

\section{Random walk on the contact process}\label{sec contact Markus}

\subsection{Preliminaries}\label{sec contact properties}

A c\`adl\`ag version of the contact process can be constructed from a graphical representation in the following standard way. For this, let $H:= (H(x))_{x \in \mathbb{Z}^d}$ and $I:=(I(x,e))_{x,e\in\mathbb{Z}^d \colon \norm{e}_1=1}$ be two independent collections of i.i.d Poisson processes with rate $1$ and $\lambda$, respectively.
On $\integers^d \times [0,\infty)$, draw the events of $H(x)$ as \emph{crosses} over $x$ and the events of $I(x,e)$ as \emph{arrows} from $x$ to $x+e$.

For $x,y \in \integers^d$ and $0 \leq s \leq t$, we say that $(x,s)$ and $(y,t)$ are connected, written $(x,s) \leftrightarrow (y,t)$, if and only if there exists a directed path in $\integers^d \times [0,\infty)$ starting at $(x,s)$, ending at $(y,t)$ and going either forward in time without hitting crosses or ``sideways'' following arrows in the prescribed direction. For $A \subset \mathbb{Z}^d$ and $s \in [0,\infty)$, define the set at time $t>s$ connected to $(A,s)$ in the graphical representation by
\begin{equation}\label{eq Ctilde}
\mathcal{C}_t(A,s) := \set{ y \in \integers^d \colon \text{ there exist } x\in A \text{ such that } (x,s) \leftrightarrow (y,t)}.
\end{equation}
When $A = \{x\}$ for some $x \in \mathbb{Z}^d$, we write $C_t(x,s)$ for simplicity. 
Note that this construction allows us to couple copies of the contact processes starting from different configurations. For each $A \subset \integers^d$ denote by $(\xi_t^A)_{t \geq 0}$ the process with initial configuration $\xi_0^A(x) =\ind_A$ and, for all  $y\in \mathbb{Z}^d$ and $t > 0$,
\begin{equation} \xi_t^A(y) = \ind_{\mathcal{C}_t(A,0)}(y) .\end{equation}
Let $(\xi_t^{A,\lambda})_{t \geq 0}$ denote the contact process with starting configuration $A \subset \mathbb{Z}^d$ and infection parameter $\lambda>0$.

The following monotonicity property is a direct consequence of the graphical construction. 
\begin{lem}[Monotonicity property]\label{lem con mon}
The contact process $(\xi_t^{A,\lambda})_{t \geq 0}$ has an attractive graphical representation coupling, which is stochastically monotone in $A$ and in $\lambda$.
\end{lem}

We next recall  the self-duality property which is used in the proof of Theorem \ref{prop contact process}a) and b). 
For this, define the backwards process $(\hat{\xi}_s^{A,t})_{0 \leq s \leq t}$ given $A \subset \mathbb{Z}^d$ and $t>0$ by
\begin{equation}
\hat{\xi}_s^{A,t}(x) = 
\begin{cases} 1\qquad& \text{if there exists $y\in A$ such that }  y \in \mathcal{C}_{t}(x,t-s);\\
0&\text{otherwise.}\end{cases}
\end{equation}
Then, the distribution of $(\hat{\xi}_s^{A,t}(x))_{s \geq 0}$ is the same as that of the contact process with the same initial configuration. Moreover, the backwards process and the contact process satisfy the duality equation. Namely,
\begin{equation} \xi_t^A \cap B \neq \emptyset \text{ if and only if } A \cap \hat{\xi}_t^{B,t} \neq \emptyset, \quad \text{ for any } A,B \subset \integers^d.\end{equation}

\subsection{Proof of Theorem \ref{prop contact process}a)}\label{sec contact coupling}

Theorem \ref{prop contact process}a) for the contact process started from $\bar{1}$ is an immediate consequence of 
Theorem \ref{thm main} and Lemma \ref{lem con mon}. We next show how to extend this to all measures stochastically dominating $\bar{\nu}_{\lambda}$. The proof goes by showing that \eqref{eq conemixing} in Theorem \ref{prop coupling speed} holds. Actually, our proof of \eqref{eq conemixing} is more general and applies to the contact process started from any measure stochastically dominating  a non-trivial Bernoulli-product measure. For this, we first state and prove two lemmas.

\begin{lem}
\label{lem dual}
For $\lambda > \lambda_c(d)$  there exist constants $a,C,c>0$ such that for all $t >0$,
\begin{equation}
\widehat{P} \left( |\mathcal{C}_t(o,0)| \leq at \mid \mathcal{C}_t(o,0) \neq \emptyset \right) \leq Ce^{-ct},
\end{equation}
\end{lem}

\begin{proof}
By  \cite[Theorem I.2.30]{LiggettSIS1999}, for some constants $C,c_1>0$,
\begin{equation}\label{eq lemma 4.2 help me again}
	\widehat{P} \left(\sup _{s\ge0} |\mathcal{C}_s(o,0)| <\infty \mid \mathcal{C}_t(o,0) \neq \emptyset \right) \leq Ce^{-c_1t}.
\end{equation} 
Let $c_2>0$ be the number such that $e^{-c_2}=1/(1+\lambda)$, where $e^{-c_2}<1$ is the probability 
that an occupied site becomes vacant before producing any offsprings.
The probability on the left hand side of \eqref{eq lemma 4.2 help me again} is clearly bounded from below by the probability that $|\mathcal{C}_t(o,0)|\leq at$ and each of the particles in $\mathcal{C}_t(o,0)$ dies before producing further offsprings. Since the contact process is Markovian and attractive, 
\begin{equation} \label{eq lemma 4.2 help me again2}
	\widehat{P} \left(\sup _{s\ge0} |\mathcal{C}_s(o,0)| <\infty \mid \mathcal{C}_t(o,0) \neq \emptyset \right) 
	\geq \widehat{P} \left(|\mathcal{C}_t(o,0)| \leq at \mid \mathcal{C}_t(o,0) \neq \emptyset \right) e^{-c_2at}.
\end{equation}
Combining the two bounds  \eqref{eq lemma 4.2 help me again} and  \eqref{eq lemma 4.2 help me again2}, we obtain the result for $c=c_1-ac_2$, and $c>0$ if and only if $a<c_1/c_2$. 
\end{proof}

Using the duality property together with Lemma \ref{lem dual} we obtain the following estimate.

\begin{lem}[Coupling of the contact process]\label{lem coupling of CP}
Let $\mu_{\rho} \in \mathcal{P}(\Omega)$ be the Bernoulli product measure with density $\rho>0$. For the contact process with $\lambda>\lambda_c(d)$, there exist constants $C,c>0$ such that
\begin{equation}\label{eq coupling of CP}
\widehat{P} \left( \xi_t^{(\mu_{\rho})}(o) \neq \xi_t^{(\delta_{\bar{1}})}(o) \text{ for some } t \in [T,T+1)  \right) \leq C e^{-cT}.
\end{equation}
\end{lem}

\begin{proof}
By attractiveness, $\xi_T^{(\mu)}(o) \leq \xi_T^{(\bar{1})}(o)$ for all $\mu \in \mathcal{P}(\Omega)$. In particular, we have that $\xi_T^{(\mu_{\rho})}(o) \neq \xi_T^{(\bar{1})}(o)$ if and only if the connected set of the dual process at time $T$ started at $(o,T)$, denoted by $\hat{\mathcal{C}}_T(o,T)$, is non-empty and $\xi_0^{(\mu)}(x) = 0$ for all $x \in \hat{\mathcal{C}}_T(o,T)$.
That is, 
\begin{equation}\label{eq help contact}
\widehat{P} \left( \xi_T^{(\mu_{\rho})}(o) \neq \xi_T^{(\delta_{\bar{1}})}(o) \right) = \widehat{P} \left( \set{\hat{\mathcal{C}}_T(o,T)\neq \emptyset} \cap \set{ \xi_0^{(\mu_{\rho})}(x) =0 \: \forall \: x \in \hat{\mathcal{C}}_T(o,T)} \right). 
\end{equation}
By Lemma \ref{lem dual} 
and self-duality of the contact process, we can estimate the size of $\hat{\mathcal{C}}_T(o,T)$. In particular, for certain constants $C_1,c_1,C_2,c_2>0$,
\begin{align*}
&\widehat{P}\left( \xi_0^{(\mu_{\rho})}(x) =0 \: \forall \: x \in \hat{\mathcal{C}}_T(o,T) \mid  
\hat{\mathcal{C}}_T(o,T)\neq \emptyset \right) \\
&\leq \widehat{P} \left( |\hat{\mathcal{C}}_T(o,T) \cap B(o,rT)| \leq aT  \mid  
\hat{\mathcal{C}}_T(o,T)\neq \emptyset \right) \\
&\qquad + \widehat{P} \left({ \xi_0^{(\mu_{\rho})}(x) =0 \: \forall \: x \in \hat{\mathcal{C}}_T(o,T)}  \mid |\hat{\mathcal{C}}_T(o,T) \cap B(o,rT)| \geq aT \right)
\\ 
&\leq  C_1e^{-c_1T} + C_2e^{-c_2T}.
\end{align*}
Thus, the l.h.s.\ of \eqref{eq help contact} decays exponentially (in $T$). To conclude  \eqref{eq coupling of CP}, by the graphical representation, it is sufficient to control the times at which there is an arrow events $I(e,o)$ from $e \in \mathbb{Z}^d$ with $\norm{e}_1=1$. Note that the number of such events is  Poisson distributed with parameter $2d\lambda$. In particular, 
\[ \widehat{P} \left( \xi_t^{(\mu_{\rho})}(o) \neq \xi_t^{(\delta_{\bar{1}})}(o) \text{ for some } t \in [T,T+1) \right) \leq (1+2d\lambda) \left(C_1e^{-c_1T} + C_2e^{-c_2T}\right).\]
\end{proof}

The proof of Theorem \ref{prop contact process}a) follows as a consequence of Lemma \ref{lem coupling of CP} and Theorem \ref{prop coupling speed}.

\begin{proof}[Proof of Theorem \ref{prop contact process}a)]
 Let $\mu_{\rho} \in \mathcal{P}(\Omega)$ be a non-trivial Bernoulli product measure, let $m,T \in (0,\infty)$ and consider $V_m(T) := V_m \cap \left( \integers^d \times [T,T+1) \right)$ with $V_m$  as defined in \eqref{eq cone m}. Since the contact process is attractive and translation invariant,
\begin{align*}
&\widehat{P} \left( \exists (x,t) \in V_m(T): \: \xi_t^{(\mu)}(x) \neq \xi_t^{(\delta_{\bar{1}})}(x) \right)
\\  \leq &(m(T+1))^d \widehat{P} \left( \xi_t^{(\mu)}(o) \neq \xi_t^{(\delta_{\bar{1}})}(o) \text { for some }t \in [T,T+1) \right).
\end{align*}
Hence, by Lemma \ref{lem coupling of CP}, for some constant $C >0$,
\begin{align*}
&\widehat{P} \left( \exists (x, t) \in V_m \cap \left( \integers^d \times [T,\infty) \right) \colon \xi_t^{(\mu)}(x) \neq \xi_t^{(\delta_{\bar{1}})}(x) \right)
\\ \leq &C \sum_{k =0}^{\infty} (T+  k)^de^{-c(T+ k)}, 
 \end{align*}
and this vanishes as $T \rightarrow \infty$.
Hence, since $m$ was arbitrary chosen, the contact process started from $\mu_{\rho}$ satisfies equation \eqref{eq conemixing} for any $m \in [0,\infty)$. Evoking Theorem \ref{prop coupling speed} this  completes the proof of Theorem \ref{prop contact process}a), noting that $\bar{\nu}_{\lambda}$ stochastically dominates a non-trivial Bernoulli product measure, as shown in \cite[Corollary 4.1]{LiggettSteifSD2006}, and by using Lemma \ref{lem con mon}.
\end{proof}

\subsection{Proof of Theorem \ref{prop contact process}b)}\label{sec contact a}
In the remaining part of this article we present the proof of Theorem \ref{prop contact process}b). We start by showing that $\rho(\cdot)$ is non-decreasing and right-continuous.

\begin{proof}[Proof of monotonicity and right continuity of $\lambda\mapsto\rho(\lambda)$]

Monotonicity of $\lambda\mapsto\rho(\lambda)$ follows directly by the coupling construction in Section \ref{sec construction}  and the graphical representation of the contact process, Lemma \ref{lem mono} and Lemma \ref{lem con mon}.

For right-continuity, let $\lambda \in (0,\infty)$ and denote by $\xi(\lambda)$ the corresponding contact process.
For $T>0$, let
\begin{equation}
f(T,\lambda) := \frac{1}{T} \widehat{\mathbb{E}}_{\delta_{\bar{1}},o}\big[\rho_{T}(\xi(\lambda)) \big].
\end{equation}
By Theorem \ref{cor main}, it follows that $f(T,\lambda) \downarrow \rho(\lambda)$ as $T \rightarrow \infty$. Moreover, by Lemma \ref{lem mono} and Lemma \ref{lem con mon}, $f(T,\lambda)$ is also non-decreasing in $\lambda$. 
Hence, $\lambda\mapsto\rho(\lambda)$ is right-continuous as the decreasing limit of non-decreasing continuous functions provided  that $\lambda\mapsto f(T,\lambda)$ is continuous for any fixed $T>0$. 

To see that $f(T,\lambda)$ is continuous in $\lambda$, note first that in order to determine the behaviour of $W_t$ for $t \in [0,T]$ we only need to consider the contact process in a finite space-time box. This follows by large deviation estimates on $N_T$ and our restriction on the transition rates, i.e.\  $\norm{u_0}_1,\norm{u_1}_1<\infty$. By the weak law of large numbers, this suffices to conclude that the probability of the walker escaping a box of size $L$ within time $T$ converges to $0$ as $L \rightarrow \infty$. Continuity of $f(T,\lambda)$ now follows by using the graphical representation of the contact process and standard arguments for functions on the contact process in a finite space-time region (see e.g. the discussion on page 40 in \cite{LiggettSIS1999}).
\end{proof}

We continue with the proof of \eqref{eq nontrivial rho}. It is not difficult to see that $\rho(\lambda) <1$ for any $\lambda \in (0, \infty)$, since vacant sites appear independently. This was already observed in den Hollander and dos Santos \cite{HollanderSantosRWCP2013} for $d=1$, 
  and their argument transfers directly to higher dimensions.

However, in order to show that $\rho(\lambda) >0$, the arguments in \cite{HollanderSantosRWCP2013} do not carry over. 
In brief, they essentially use that there is a positive density of ``waves'' of particles moving from the right to the left in space-time. Using the ordering of $\integers$ and monotonicity in the displacement of a nearest neighbour random walk, the random walk cannot escape less than a positive proportion of the waves. Due to the one-dimensional nature of this argument it does not carry over to general dimensions nor does it extend beyond nearest neighbour jumps.

Our proof is based on monotonicity of $\rho_{t}(\xi)$.
We propose here a simple strategy which also generalises to many other monotone dynamics. In particular, by applying the theory in \cite{BezuidenhoutGrayCASS1994}, our argument for $d\geq 2$ seems to extend to all super-critical additive spin-flip systems. 
For dimension $d=1$, we use an improved version of the corresponding proof in \cite{HollanderSantosRWCP2013}. For both the one-dimensional and higher dimensional cases we make use of the general construction in Section \ref{sec construct general}.

\begin{proof}[Proof of \eqref{eq  nontrivial rho} for $d=1$]

 Let $\lambda > \lambda_c(1)$ and consider the contact process $(\xi_t)$ on $\integers$ with infection parameter $\lambda$ and initial configuration drawn according to $\mu \in \mathcal{P}(\Omega)$, where $\mu$ is the distribution of \[ \eta = \eta' \cdot \ind_{(-\infty,0)}\text{, where }\eta' \sim \bar{\nu}_{\lambda}.\] 
 Further, for $0\leq s \leq t$ and $z\in \integers$, denote by
\begin{equation}
r_{s,t}(z) := \sup \{ y \in \integers \colon  \xi_s(x)=1 \text{ for some } x \leq z \text{ and } (x,s) \leftrightarrow (y,t) \}
\end{equation}
the rightmost site that is occupied at time $t$  by a particle and connected to a site to the left of $z$ at time $s$. It is well known that there exists $\alpha>0$ (depending on $\lambda$) such that
\begin{equation}\label{eq rightmost speed}
\lim_{t \rightarrow \infty} \frac{1}{t} r_{0,t}(o) = \alpha \quad P^{\mu}-a.s.
\end{equation}
See \cite[Theorem VI.2.19]{LiggettIPS1985}  for a proof of \eqref{eq rightmost speed} with respect to $P^{\delta_{\bar{1}}}$.
The extension to $P^{\mu}$ follows from this statement and standard coupling arguments (e.g.\ \cite[Theorem VI.2.2]{LiggettIPS1985}). 

(Recall the general construction of $(W_t)$ in Section \ref{sec construct general}). We next specify the family of Boolean functions $(f_k)_{k \geq 0}$, which is defined by an iterative procedure involving a second process  of elements in $\integers$ denoted by $(R_k)$. 
 Assume w.l.o.g.\ that $u_0 \leq 0$.
Let  $R_0= r_{0,J_0}(o)$ be the position of the rightmost particle of the contact process at the first jump time of the random walk. Recall that $S_0=o$, and define iteratively for $k \geq 0$,
\begin{align}
f_k:= &\begin{cases} \xi_{J_{k}}(S_{k}) \qquad&\text{if } S_{k} \leq R_{k};\\
 0&\text{otherwise}.\end{cases}
 \\ R_{k+1}:=&\begin{cases} r_{J_{k},J_{k+1}}(S_{k}) \qquad &\text{if } S_{k} \leq R_{k};
\\ r_{J_{k}, J_{k+1}} (R_{k}) &\text{otherwise}.\end{cases}
\end{align}
Hence, if $r_{0,J_0}(o) \geq 0$, then $R_1$ is assigned the position of the rightmost particle which at time $J_1$ is connected to an occupied site to the left of  $o$ at time $J_0$, the first jump time. Further, in this case, $f_0$ is assigned the value of the contact process at the location of the random walk  at the $1$'st jump time. Thus, the random walk ``observes'' the environment and chooses its transition kernel accordingly.
Otherwise, if $r_{0,J_0}(o) < 0$, $R_1$ is assigned the position of the rightmost particle which at time $J_1$ is connected to an occupied site to the left  of $r_{0,J_0}(o)$ at time $J_0$ and  $f_0$ is assigned the value $0$. Consequently, for this latter case, the random walk jumps as if it had observed a vacant site.

For arbitrary $k \geq 0$, if $S_{k} \leq R_{k}$, then $R_{k+1}$ is assigned the location of the rightmost site at  the $(k+2)$'th jump time which is connected to an occupied site to the left of $S_{k}$ at the $(k+1)$'th jump time, and in this case $f_k$ is assigned the value $\xi_{J_{k}}(S_{k})$. 
On the other hand, if $S_{k-1} > R_{k-1}$, $R_{k+1}$ is a prolongation of $R_{k}$  and we set $f_k=0$. 
By construction, and using that $\mathbb{Z}$ is ordered and the contact process has nearest neighbour interactions, the following holds: at times $k$ for which $S_k \leq R_k$ there is a connected path $(\omega_t)_{0\leq t \leq J_{k}}$ such that $\xi_0(\omega_0)=1$, $\omega_l \leq S_l$ for all $l \leq k-1$ and $\omega_k \geq S_k$. 
Furthermore, since $f_k$ is the product of an indicator function and $\xi_{J_{k}}(S_{k})$, we have  $\rho^{(1)}(k, \xi)\leq \rho(k, \xi)$, where $\rho^{(1)}(k,\xi)= \sum_{i=0}^{k-1} f_i$ and $\rho(k,\xi)$ is as in Section \ref{sec construction}. By the last observation, in order to prove the first part of \eqref{eq nontrivial rho}, it is sufficient to show that there exists $\rho^{(1)}>0$, depending on $\lambda$, such that
\begin{equation}\label{eq positive density 1}
\liminf_{k \rightarrow \infty} \frac{1}{k} \rho^{(1)}(k,\xi) \geq \rho^{(1)}, \quad \widehat{\Prob}_{\mu,o}-a.s.
\end{equation}
 For this, let $T_0=0$ and, for $k \geq 1$, let
 \begin{equation}
 T_k := \inf \{ n > T_{k-1} \colon S_{n-1} \leq R_{n-1}\}
 \end{equation}
denote the $k$'th time that the random walk observes the environment and set 
 \begin{equation}\label{eq tau}
 \tau_k := T_k - T_{k-1}, \qquad k\ge1.
 \end{equation}
 As noted above, at times $T_k$, 
  there is a connected path $(\omega_t)_{0\leq t \leq J_{k}}$ such that $\xi_0(\omega_0)=1$, $\omega_l \leq S_l$ for all $l \leq k-1$ and $\omega_k \geq S_k$. 
 At such times, the law of $(\xi_{J_{k-1}}(x))_{x \leq z}$ stochastically dominates $\bar{\nu}_{\lambda}$, as shown in  \cite[Lemma 4.1]{HollanderSantosRWCP2013}. Consequently, since $\tau_k$ is decreasing with respect to the configuration of the contact process on sites strictly to the left of $S_{T_{k-1}}$, the  times $(\tau_k)_{k \geq 1}$  are dominated by an i.i.d.\ sequence of $\tau_0$ distributed random variables. Furthermore, for the same reason, at times $T_k$, the random walk has a probability of at least  $\bar{\nu}_{\lambda}(\eta(o)=1)>0$ of observing an occupied site. 
Thus, the first part of  \eqref{eq nontrivial rho} follows once we have shown that 
\begin{equation}\label{eq tau finite mean}
\widehat{\mathbb{E}}_{\nu_{\lambda, o}}(\tau_0) < \infty.
\end{equation}
For this, note that, since $\{ \tau_0\geq n\} = \{ S_{k-1} \geq R_{k-1} \text { for all } k \leq n\}$,
\begin{align}\label{eq tau finite mean help}
\widehat{\Prob} \left( \tau_0 \geq n \right) \leq \widehat{\Prob} \left( S_{n-1} \geq \beta (n-1), \tau_0\geq n \right) + \widehat{\Prob} \left( R_{n-1} \leq \beta (n-1) \right), 
\end{align}
for any $\beta>0$. For any $\beta <\alpha$ (with $\alpha$ as in \eqref{eq rightmost speed}), the rightmost term in \eqref{eq tau finite mean help} decays exponentially (in $n$) due to large deviation estimates for $r_{0,t}(o)$; see \cite[Corollary 3.22]{LiggettIPS1985}. Moreover, for $\beta>0$, the leftmost term of \eqref{eq tau finite mean help} decays like $n^{-2}$ as  $n \to \infty$. This follows  by applying Chebyshev's inequality, using that $(W_t)$ has finite second moments, and since, under $\tau_0>n$, the random walk $(S_k)_{0\leq k \leq n}$ behaves as a simple random walk in a $0$-homogeneous environment. Consequently, by setting $0<\beta <\alpha$, \eqref{eq tau finite mean} holds and this concludes the first part of  \eqref{eq nontrivial rho}. 

To conclude the  second part of \eqref{eq nontrivial rho}, we note that $\alpha=\alpha(\lambda)$ in \eqref{eq rightmost speed} diverges to $\infty$ as $\lambda \to \infty$. In particular, by reasoning as above and choosing $\beta = \alpha/2$, we have that  $ \lim_{\lambda \rightarrow \infty} \widehat{\Prob} \left( \tau_0 \geq 2 \right)=0$.  Moreover, since $\lim_{\lambda\rightarrow\infty} \bar{\nu}_{\lambda}(\eta(o)=1)=1$, it follows that $\rho^{(1)}$ in \eqref{eq positive density 1} approaches $1$ as $\lambda \rightarrow \infty$ from which, by the remark above \eqref{eq positive density 1}, we  conclude the proof.
  \end{proof}

\textit{We proceed with the proof of \eqref{eq nontrivial rho} for the case $d\geq 2$.} For this, we make use of the fact that the supercritical contact process survives in certain space-time slabs, as first shown in  \citet{BezuidenhoutGrimmettCP1990}. For the cases when $u_0\neq o$ their result suffices. However, in order to also treat the special case when $u_0 =o$, we use an extension of their theorem to certain tilted slabs. 

For $K \in \nat$, $L \in \reals$ and $A\subset \mathbb{Z}^d$, denote by $\left(_K^L \xi_t^A\right)_t$ the \emph{truncated} contact process defined via the graphical representation by
\begin{equation}
_K^L \xi_t^A(x) = \begin{cases} 1 \quad&\text{if }\{ (y,0) \colon y \in A\} \text{ is connected to } (x,t)
\text{ within } \mathcal{S}_{K,L};\\ 0 &\text{otherwise}.\end{cases}
\end{equation}
Here, 
\begin{equation}\label{def tilted slab}
\mathcal{S}_{K,L} := \left\{ (x,t) \in \integers^d \times [0,\infty) \colon \: x \in [-K,K] \times \integers^{d-1} + [Lt,0,\dots, 0] \right\},
\end{equation}
and $(A,0)$  is connected to $ (x,t)$
 within  $\mathcal{S}_{K,L}$ if $(A,0)$ is connected to $(x,t)$ in the graphical representation without using arrows outside $\mathcal{S}_{K,L}$.
We say that $\mathcal{S}_{K,L}$  is a tilted slab if $L \neq 0$, and that it is not tilted if $L=0$.

\begin{prop}[Survival in tilted slabs]\label{prop Grimmett1}
Let $d\geq 2$. 
\begin{description}
\item[i)] For  $\lambda > \lambda_c(d)$, there exist $K(\lambda) \in \nat$, $L(\lambda) > 0$ 
such that for all $K> K(\lambda)$ and $L \in (-L(\lambda),L(\lambda))$,
\begin{equation}\label{eq tilting} \widehat{P} \left( _K^L\xi_t^{\set{o}} \neq \emptyset \quad \forall \: t  \geq 0 \right) > 0. \end{equation}
\item[ii)] There exists $K \in \nat$ such that for all $L>0$,
\begin{equation}\label{eq Markus survival of the fittest}
\lim_{\lambda \rightarrow \infty} \widehat{P} \left( _{K}^L\xi_t^{\set{o}} \neq \emptyset \quad \forall \: t \geq 0 \right) = 1.
\end{equation}

\end{description}
\end{prop}

The proof of survival in non-tilted slabs proceeds via a block argument and comparison with a certain (dependent) oriented percolation model. As pointed out by \citet{BezuidenhoutGrimmettCP1990}, there is a certain freedom in the spatial location of these blocks. The proof of Proposition \ref{prop Grimmett1} is achieved by adapting the proof of \cite{BezuidenhoutGrimmettCP1990} in a way where the blocks are organised in a tilted way.
A sketch of the proof of Proposition \ref{prop Grimmett1} is given at the end of this section.

Using that infections can spread fast in a small time interval, we note that also the following corollary of Proposition \ref{prop Grimmett1} holds. 
\begin{cor}\label{cor Grimmett1}
Let $\lambda > \lambda_c(d)$ and consider the same parameters as in Proposition \ref{prop Grimmett1}. Then, for any $\delta>0$ small enough there is an $\epsilon>0$, depending on all parameters, such that for all $(x,t) \in S_{K,L}$ with $(x,t-\delta) \in S_{K,L}$, we have that  
$\widehat{P} \left( _K^L\xi_t^{[-K,K]\times \mathbb{Z}^{d-1}}(x)=1 \right) > \epsilon,$
and $\epsilon=\epsilon(\lambda)$ approaches $1$ as $\lambda \rightarrow \infty$.
If $L=0$, then the claim also holds for $\delta=0$.
\end{cor}

\begin{proof}
First note that, since the process is started from all sites in $[-K,K]\times \mathbb{Z}^{d-1}$ occupied and the space-time slab  $\mathcal{S}_{K,L}$ is translation invariant in all coordinate directions besides the first one, we may assume w.l.o.g.\ that $x=(x_1,0,\dots,0)$. Further, by Proposition \ref{prop Grimmett1}i), we have that 
\begin{align*}\widehat{P} \left( _K^L\xi_{t-\delta}^{[-K,K]\times \mathbb{Z}^{d-1}}(y)=1 \right) > 0 \text{ for some }y \in [-K+ L(t-\delta),K +L(t-\delta)] \times \{0,\dots,0\}.\end{align*} The claim thus follows by the Markov property and since an infection at $(y,t-\delta)$ can spread to each $(x,t)$ with $x \in [-K +Lt, K+Lt] \times \{0,\dots, 0\}$ with positive probability. In particular, since $x$ is such that  also $(x,t-\delta) \in S_{K,L}$ and the set $[-K +Lt,K +Lt] \times \{0,\dots, 0\}$ is finite,  this probability is uniformly bounded away from $0$.
\end{proof} 
We next present the proof of \eqref{eq nontrivial rho} in Theorem \ref{prop contact process} for the case $d\geq 2$, assuming Proposition \ref{prop Grimmett1} to be true.

\begin{proof}[Proof of \eqref{eq  nontrivial rho} for $d \geq 2$]
 We consider first the case where $u_0 \neq o$, for which we do not need the notion of tilted slabs in order to prove the l.h.s.\ of \eqref{eq  nontrivial rho}. Moreover, in this case, by translation invariance of the contact process, we assume w.l.o.g.\ that $u_0 \cdot e_1 <0$. Let $\lambda > \lambda_c(d)$, and  
let $K \in \nat$ be  such that  Proposition \ref{prop Grimmett1}i) is satisfied with $L=0$. Partition $\integers^d\times [0,\infty)$ into slabs $\Pi_i = 2Ki + \mathcal{S}_{K,0}$, $i\in\integers$, 
and consider $(\zeta_t^{(i)})_{i \in \integers}$ consisting of independent copies of  the process $(_K^0 \xi_t^{[-K,K]\times \mathbb{Z}^{d-1}})$. Further, denote by $(\zeta_t)$ the process on $\Omega$ where, for $(x,t) \in \Pi_i$, we set $\zeta_t(x) = \zeta_t^{(i)} (\theta_{2Ki\cdot e_1}x)$. 

We next specify the family of Boolean functions $(f_k)_{k \geq 0}$. 
Let $R_0=1$, recall that $S_0=o$, and define iteratively for $k \geq 0$,
\begin{align}
f_k:= &\begin{cases} \zeta_{J_{k}}(S_{k}) \qquad&\text{if } (S_{k},J_{k}) \in \Pi_i \text{ for some } i< R_{k};\\
 0&\text{otherwise}.\end{cases}
\\ R_{k+1}:= &\begin{cases} i \qquad&\text{if } (S_{k},J_{k}) \in \Pi_i \text{ for some } i< R_{k};\\
 R_k&\text{otherwise}.\end{cases}
\end{align}
That is, $f_0= \zeta_{J_0}(o)$ and $R_1=0$. Further, for arbitrary $k$, $R_k$ records the label of the leftmost slab the random walk has ``observed'' at jump times $J_0,\dots, J_{k}$. If, at a jump time, the random walk  finds itself inside a slab which is at the left of all the slabs it previously has observed, then, by the definition of $f_k$, the random walk ``observes'' the environment. Otherwise, $f_k=0$, and the random walk acts as if it had seen a vacant site.
In particular, 
 we have that $\rho^{(d)}(k, \zeta)\leq \rho(k, \zeta)$, where $\rho^{(d)}(k,\zeta)= \sum_{i=0}^{k-1} f_i$ and $\rho(k,\zeta)$ is as in Section \ref{sec construction}. As in the $d=1$ case, and since $(\zeta_t) \leq (\xi_t)$, it is  sufficient to show that there exists $\rho^{(d)} >0$ such that
\begin{equation}\label{eq positive density 2}
\liminf_{k \rightarrow \infty} \frac{1}{k} \rho^{(d)}(k,\zeta) \geq \rho^{(d)} \quad \widehat{\Prob}_{\delta_{\bar{1}},o}-a.s.
\end{equation}
For this, let $T_0=0$ and, for $k \geq 1$, let
 \begin{equation}
 T_k := \inf \{ n > T_{k-1} \colon  R_n< R_{n-1}\}
 \end{equation}
denote the $k$'th jump time at which the random walk is in a slab to the left of the origin which it previously has not observed, and let  
 \begin{equation}\label{eq tau}
 \tau_k := T_k - T_{k-1}, \quad k\ge1,
 \end{equation}
be  the number of jumps it takes before the random walk observes a new slab.
By Corollary \ref{cor Grimmett1} (which holds with $\delta=0$ since $L=0$), at the times $(\tau_k)$, the random walk has a positive probability to observe an occupied site, uniform in $k$. We conclude \eqref{eq positive density 2}, since the times $\tau_k$ have finite mean, uniformly in $k$. Indeed, the latter follows since $u_0 \cdot e_1 < 0$ and by Chebyshev's inequality, using that $(W_t)$ has finite second moments.\medskip

We continue with the case $u_0 =o$, for which we need to make certain modifications to the approach above. Choose $L>0$ and $K \in \nat$ such that Proposition \ref{prop Grimmett1}i) holds and partition $\mathbb{Z}^d \times [0,\infty)$ into tilted slabs $\tilde{\Pi}_i = 2Ki + \mathcal{S}_{K,L}$. For 
$i \in \integers$, denote by $(\tilde{\zeta}_t^{(i)})_{i \in \integers}$ independent copies of  the process $(_K^L \xi_t^{[-K,K]\times \mathbb{Z}^{d-1}})$. Further, denote by $(\tilde{\zeta}_t)$ the process on $\Omega$ where, for $(x,t) \in \tilde{\Pi}_i$, we set $\tilde{\zeta}_t(x) = \tilde{\zeta}_t^{(i)} (\theta_{2Ki\cdot e_1}x)$.
Next, let $\tilde{R}_0=1$ and recall that $S_0=o$. Fix $\delta>0$ small such that Corollary \ref{cor Grimmett1} holds, and define iteratively (for $k\geq 0$),
\begin{align}
f_k:= &\begin{cases} \tilde{\zeta}_{J_{k}}(S_{k}), &\text{if } \{S_{k}\} \times [J_{k} -\delta,J_{k}) \in \tilde{\Pi}_i \text{ for some } i<\tilde{R}_{k};\\
 0, &\text{otherwise}.\end{cases}
 \\ \tilde{R}_{k+1}:= &\begin{cases} i \qquad&\text{if } (S_{k},J_{k}) \in \tilde{\Pi}_i \text{ for some } i< \tilde{R}_{k};\\
 R_k&\text{otherwise}.\end{cases}
\end{align}
Next, define the variables $\tilde{T}_k$ and $\tilde{\tau}_k$ similar to the non-tilted case, by replacing $R$ by $\tilde{R}$ and $T$ by $\tilde{T}$. 
Note that, by our choice of $\delta>0$, with strictly positive probability (uniformly in $k$), it holds that $\tilde{\zeta}_{J_{\tau_k}}(S_k)=1$. 
It hence suffices to show that the times $\tau_k$ have finite mean, which implies that \eqref{eq positive density 2} holds and thus the l.h.s.\ of \eqref{eq nontrivial rho}. 
To see that this indeed is the case, first note that, since the jump times $(N_t)$ are continuous,  each time the random walk enters a new slabs there is the possibility that it satisfies $ \{S_{k}\} \times [J_{k} -\delta,J_{k}) \in \Pi_i $. Lastly, the number of new slabs that $(W_t)$ observes is a positive fraction of its jumping times. This follows similarly as for the case that $u_0\neq o$, by using that now $u_0 \cdot e_1 = 0$ and that $L>0$. In particular, by again using Chebyshev's inequality we obtain sufficient estimates on the time it takes until a new slab is observed. By this we conclude  that the times $\tau_k$ have finite mean, and thus the l.h.s.\  of \eqref{eq nontrivial rho}.  \medskip

Note that the approach with tilted slabs also applies in the case when $u_0\cdot e_1<0$ (with the same $K$ and $L$). In order to  conclude $\rho(\lambda) \rightarrow 1$ as $\lambda \rightarrow \infty$ for both cases, we argue based on this approach. First note that, by Proposition \ref{prop Grimmett1}ii), we may take $L$ large (keeping $K$ fixed) by choosing $\lambda$ large enough. Consequently, we may chose $L$ large and $\delta>0$ small so that the times $(\tilde{\tau}_k)$ have mean bounded from above by $1+c$, uniformly in $k$ and for any fixed $c>0$. Subsequently, for each such $K,L$ and $\delta$, the probability $\tilde{\zeta}_{J_{\tau_k}}(S_k)=1$ can be made arbitrary close to $1$, uniformly in $k$, by again tuning $\lambda$ large. This yields the proof of \eqref{eq nontrivial rho} for $d\geq 2$ and thus concludes the proof of Theorem \ref{prop contact process}.
\end{proof}


We end with a sketch of the proof of Proposition \ref{prop Grimmett1}.

\begin{proof}[Sketch of the proof of Proposition \ref{prop Grimmett1}]
For the proof of Proposition \ref{prop Grimmett1} we adapt the proof of  \cite[Theorem 1.2.30]{LiggettSIS1999}. The idea is to proceed in the same manner, however, instead of  comparing the survival of the process with the ordinary oriented percolation structure, we compare it to a certain tilted oriented percolation model.

To make this more precise, consider the following sub-graph of $\integers^d \times \integers$. Fix $l \in \nat$. Next, set $V_0 = \set{(0,0,\dots,0)}$ and, for $n \in \nat$, iteratively define 
\[V_{n} =\left\{
\begin{array}{cc} V_{n-1} + \set{(0,l,0,\dots 0,l)} & \text{ if }n\text{ is odd}, \\V_{n-1} + \set{(1,0,\dots ,0,1),(-1,0,\dots ,0,1} & \text{ if }n\text{ is even}.\end{array}\right. \]

Let the directed graph $G=(V,E)$ be given by $V= \bigcup_{i=1}^{\infty} V_i$, and, for any pair $x,y \in V$, the (directed) edge $(x,y) \in E$ if and only if $x \in V_{n-1}$ and $y \in V_n$ for some $n \in \nat$ with
\[ y = \left\{\begin{array}{cc}x + (0,l,0,\dots,0,l) & \text{ if }n\text{ is odd}, \\x + (\pm 1, 0, \dots,0, 1) & \text{ if }n\text{ is even}.\end{array}\right. \]
This produces a tilted oriented percolation graph where the tilting in the second coordinate depends on the ratio $\frac{l}{l+1}< 1$. As for the ordinary oriented percolation model, let each edge be open with probability $p$ independently, otherwise closed. By  \cite[Theorem B26]{LiggettSIS1999}, 
when $p$ is large enough (depending on $l$), the origin lies in an infinite connected component of open edges with positive probability.

The proof of the first part of Proposition \ref{prop Grimmett1} now follows similar as the proof of Theorem 2.23 in \cite{LiggettSIS1999}, by constructing a coupling between the contact process and the oriented percolation model on the above defined graph $G$. Thus, by choosing $\epsilon>0$ in \cite[Proposition 2.22]{LiggettSIS1999}  small enough,  with positive probability $G$ percolates when  edges are open with probability $p=1-\epsilon$. By the coupling construction in \cite{LiggettSIS1999}, this also holds for the contact process, depending only on the graphical representation within $\mathcal{S}_{K,L}$, with
 $K= 5a+2a\frac{l}{l+1}$ and $L = \frac{l}{l+1} \frac{2a}{5b}$ . This completes the proof of the first part.   

The second part follows by monotonicity in $\lambda$. By a standard coupling procedure we may couple the systems with infection rates $\lambda$ and $3\lambda$ such that, if  \cite[Proposition 2.22]{LiggettSIS1999} is true for $\lambda$ and constants $a$ and $b$, then it also holds for the system with infection rate $3\lambda$ and constants $a$ and $b/3$. Hence, given the constants $a$ and $b$ for a fixed $\lambda$, by letting $\lambda$ converge towards infinity we may take $b$ as small as we wish. In particular, we may choose $L$ large and still satisfy Equation \eqref{eq tilting}. On the other hand, for fixed $K,L >0$, the probability in \cite[Proposition 2.22]{LiggettSIS1999} converges to $1$ as $\lambda \rightarrow \infty$, and therefore \eqref{eq Markus survival of the fittest} follows. 
\end{proof}

\subsection*{Acknowledgments}
The authors are indebted to Rob van den Berg for kind support during various stages of this project.
They are further grateful to Frank den Hollander, Renato dos Santos and Florian V\"ollering for fruitful discussions, and to the referees for valuable comments and suggestions.
This work is supported by the Netherlands Organization for Scientific Research (NWO).

\end{document}